\documentclass[  headinclude,  open=right, toc=listof, toc=bibliography]{article}
\usepackage[T1]{fontenc} 
\usepackage[utf8]{inputenc}
\usepackage{amsmath}
\usepackage{amssymb}
\usepackage{amsthm}
\usepackage{dsfont}
\usepackage{mathtools}
\usepackage{enumerate}
\usepackage{mathrsfs} \usepackage{color} 	
\usepackage[backend=biber, url=false,isbn=false, doi=false, bibstyle=numeric, giveninits=true, date=year,  maxbibnames=9]{biblatex}
\usepackage{bbm}
\usepackage{xcolor}
\usepackage{graphicx}
\usepackage[normalem]{ulem}

\colorlet{colorCE}{red}
\colorlet{colorAL}{blue!50}

\usepackage[colorinlistoftodos,prependcaption,textsize=small]{todonotes} % for editing
\colorlet{colorAL}{blue!50}
\colorlet{colorCE}{red!50}

\usepackage{geometry}
\newtheorem{theorem}{Theorem}[section]
\newtheorem{definition}[theorem]{Definition}

\newtheorem{proposition}[theorem]{Proposition}
\newtheorem{lemma}[theorem]{Lemma}
\theoremstyle{definition}
\newtheorem{remark}[theorem]{Remark}

\newcommand{\R}{\mathbb{R}}
\newcommand{\C}{\mathbb{C}}

\newcommand{\T}{\mathbb{T}}

\begin{document}

\title{Operator models for meromorphic functions of bounded type}
\author{Christian Emmel}
\maketitle
\begin{abstract}
In this article, we construct operator models on Krein spaces for meromorphic functions of bounded type. This construction is based on certain reproducing kernel Hilbert spaces which are closely related to model spaces. Specifically, we  show that each function of bounded type corresponds naturally to a pair of such spaces, which extends Helson's representation theorem. This correspondence enables an explicit construction of our model, where the Krein space is a suitable sum of these identified spaces. Additionally, we establish that the representing self-adjoint relations possess a relatively simple structure, since they turn out to be partially fundamentally reducible. Conversely, we show that realizations involving such relations correspond to functions of bounded type.
\color{black}
\end{abstract}
\section{Introduction}
Quotients of bounded analytic functions, known as meromorphic functions of bounded type, were initially studied  by Nevanlinna over a century ago \cite{Nevanlinna1925}. Since then, they have gained significant importance in function theory, mostly because of  their crucial role in the theory of Hardy spaces \cite{Rosenblum1994tih}. These days, they constitute a fundamental function class in complex analysis and have applications in many different fields, for example in systems theory (see e.g. \cite{Inouye1986}) and arithmetic geometry (see e.g. \cite{Fritz}).
\par\smallskip
Our main result establishes that functions of bounded type can be represented by \textit{operator models} on Krein spaces - indefinite inner product spaces  which decompose into an orthogonal sum of a Hilbert space and an anti-Hilbert space. 
Specifically, for a given function  $q \vcentcolon U \subset \C^+ \rightarrow \C$ of bounded type on the upper half-plane $\C^+$, we construct  self-adjoint relation $A$ in  a  Krein space $\mathcal{K}$ such that
%An operator model of $q$ consists of a  Krein space $\mathcal{K}$, a  self-adjoint relation $A$ on $\mathcal{K}$, and a generating vector $v \in \mathcal{K}$, which relate to  $q$ via the representation
\begin{gather}
U \subset\varrho(A) \quad \text{ and} \quad 
q(\zeta)= \overline{q(\zeta_0)}+(\zeta- \overline{\zeta_0})  \left[\left(I+(\zeta-\zeta_0)(A-\zeta)^{-1}\right) v, v\right]_{\mathcal K}, \quad \forall \zeta \in U \label{For:Intr1} \\
\overline{\mathrm{span}} \{ \left(I+(\zeta-\zeta_0)(A-\zeta)^{-1}\right) v  \vcentcolon \zeta \in \varrho(A) \}= \mathcal K, \label{For:Intr2}
\end{gather}
where $\zeta_0 \in U$ is some fixed base point and $v \in \mathcal{K}$ a suitable generating vector. 
 Such an operator model establishes a strong connection between the properties of the relation $A$ and the function  $q$, for which the minimality condition  \eqref{For:Intr2} is essential. 
\par\smallskip
Operator models have been an active field of research  for at least half a century \cite{KreinKac}. The simplest example  is given by the classical integral  representation for Herglotz-Nevanlinna functions, which are analytic functions that  map $\C^+$ into  $\overline{\C^+}$.   Any such function $q$ can be expressed as
\begin{equation}\label{For:Einleitung3}
q(\zeta)=a+b  \zeta +\int_\R \left( \frac{1}{t-\zeta}-\frac{t}{1+t^2} \right)\mathrm{d}\nu(t)\quad \forall \zeta\in\mathbb \C^+,
\end{equation}
where $a\in \R $, $b \in \R_{\geq 0}$ and $\nu$ is a positive measure satisfying a certain regularity condition. 
This representation is, if suitably interpreted, a minimal realization as defined above. Indeed, if we assume $b=0$ for simplicity and choose $\mathcal{K}=L^2(\nu)$, $A$ to be the multiplication operator by the independent variable and choose $v$ appropriately, then \eqref{For:Intr1} simplifies to \eqref{For:Einleitung3}. In this case, the spectrum of $A$ is given by  
\begin{align*}
    \sigma(A)=\overline{\left\{ x\in \R
    \vcentcolon 0<\liminf_{\epsilon\downarrow 0} \mathrm{Im}(q(x+\mathrm{i} \epsilon)) \right\}},
\end{align*}
which illustrates the close connection between  $A$ and $q$. 
Finally, 
it turns out that all minimal realizations on Hilbert spaces are, up to an isomorphism, of this form and therefore correspond to  Herglotz-Nevanlinna functions. 
\par\smallskip
Consequently, in order to construct operator models for more general classes of functions, it is necessary to consider Krein spaces instead of Hilbert spaces.  This more general setup allowed for the construction of minimal realizations for a wide range of functions including  
generalized Nevanlinna functions \cite{article}, \cite{edsjsr.119452920031201}, \cite{DijksmaLangerLugerShondin04}, definitizable functions \cite{Jonas}, real quasi-Herglotz functions \cite{luger2019quasiherglotz} and atomic density functions \cite{me2}.   These functions have two things in common. First, they are of bounded type,  and second, they are all
closely related to Herglotz-Nevanlinna functions. For example, generalized Nevanlinna functions are products of the form $q\cdot h$, where $q$ is a Herglotz-Nevanlinna function and $h$ is a rational function satisfying a certain symmetry condition. Since the definition of ``bounded type'' is not related to Herglotz-Nevanlinna functions in any intrinsic way, we see that the result presented in this article is  significantly more general than the preceding ones. 
\par\smallskip
Finally, we  mention the  related result obtained in \cite{me}. There  it was shown that every function $q$ of bounded type has a \emph{non-minimal} realization. This means that condition \eqref{For:Intr1} is satisfied but not the minimality condition \eqref{For:Intr2}. However, the latter is essential to connect the function theoretic properties of $q$ with the operator theoretic properties of $A$.
\par\smallskip
Next we want to say a few words about the methods used in this article. Our construction is closely connected to the class of $\mathcal{B}(h)$-spaces that de Branges studied in a series of papers in the 1960's \cite{DeBranges}, \cite{DEBRANGES196544}. 
Here, the symbol $h$ is a meromorphic function on $\C \setminus \R$ which is  bounded by $1$ on $\C^+$ and satisfies $\overline{h(\zeta)}=\frac{1}{h(\overline{\zeta})}$. We denote the set of these functions by $\mathcal{S}_0$. The induced space  $\mathcal{B}(h)$ is then the
reproducing kernel Hilbert space generated by the kernel
\begin{align*} 
    s_h(z, w)=\frac{1-h(z)\overline{h(w)}}{- \mathrm{i}  (z - \overline{w})}. 
\end{align*}
\color{black}
An important feature of such a space is that the difference-quotient operator 
\begin{align*}
    D_w \vcentcolon \mathcal{B}(h) \rightarrow \mathcal{B}(h), \quad g \mapsto  \frac{g(\cdot)-g(w)}{\cdot-w}
\end{align*}
is bounded. Moreover, we  point out that $\mathcal{B}(h)$ is (essentially) a model space 
if $h_{|\C^+}$ is an inner function.
\par\smallskip
Now let  $f$ be a given function of bounded type, which we extend as usual  by $f(\overline{\zeta})=\overline{f(\zeta)}$ to the lower half plane $\C^-$. In a first step, we show that there exist two functions $h_1, h_2 \in \mathcal{S}_0$   
such that $f$ has a  representation of the form
\begin{align}\label{Helsonsrep}
    f = \mathrm{i}  \frac{h_2-h_1}{h_2+h_1} \quad \text{ with } \quad
    \mathcal{B}(h_1)  \cap \mathcal{B}(h_2)= \{ 0\}.
\end{align}
In a second step, we prove that  
\begin{align*}
    \mathcal{L}(f) \vcentcolon = \frac{1}{h_1+h_2} \cdot \bigg(\mathcal{B}(h_1)  [+] \big(- \mathcal{B}(h_2)\big)\bigg) 
\end{align*}
is a reproducing kernel Krein space with kernel 
\begin{align*}
      \quad N_f(\zeta, w)= \frac{f(\zeta)-\overline{f(w)}}{\zeta-\overline{w}}.
\end{align*}
In a third step, we show that the difference-quotient operator $D_w$
acts boundedly on this space. Interestingly,  this follows almost directly from Leibniz's formula. Indeed,  given $g\in \mathcal{B}(h_1)$ we calculate 
\begin{align*}
    D_w\left(\frac{g}{h_1+h_2}  \right)
    =\frac{1}{h_1+h_2}   D_w(g)+g(w)  D_w\left(\frac{1}{h_1+h_2}\right).
\end{align*}
Given that  point evaluations and $D_w$ are bounded on $ \mathcal{B}(h_1)$, it is sufficient to show that  $D_w\left(\frac{1}{h_1+h_2}\right)\in \mathcal{L}(f)$. This computation is straightforward, and with this result, we can now establish our operator model. Specifically, we define  $(A-w)^{-1}=D_w$ and  $v=N_f(\cdot, \zeta_0)$. 
\par\smallskip
There are a few points to note  here. First, we emphasize  that  
only elementary facts about Krein spaces are used  in our argument. Instead,  we construct the space $\mathcal{L}(f)$ and the relation $A$  in  very explicit ways using  well-studied function spaces and Leibniz's formula. 
In contrast,  all of the previous realization theorems rely heavily on deep results from operator theory on Krein spaces. 
\par\smallskip
Second, we draw attention to the fact that the representation provided in equation \eqref{Helsonsrep} extends Helson’s representation for functions with real boundary values \cite[Theorem 3.1.]{articleRealpositive}.  When $f$ is such a function, $(h_1)_{|\C^+}$ and $(h_2)_{|\C^+}$ become inner functions, and the associated spaces $\mathcal{B}(h_1)$ and $\mathcal{B}(h_2)$ are (essentially) model spaces. In order to establish this   representation, we show that two $\mathcal{B}(h)$-spaces intersect trivially if and only if the symbols are relatively prime.  It's important to note that $B(h)$-spaces are, in general, not isomorphic to de Branges-Rovnak spaces, where an analogous statement does not hold.  In any case, these results in their full generality appear to be new, and we believe that they hold independent interest. 
\par\smallskip
Lastly, our construction of $D_w$ reveals an interesting interplay between point evaluations, difference-quotients and products of functions. We use this interplay to analyse the structure of $A$.  Specifically, it is shown that there exists a fundamental decomposition $\mathcal{L}(f)=\mathcal{K}_+ [+] \mathcal{K}_-$ such that the relations
\begin{align*}
    S_+ \vcentcolon = A \cap (\mathcal{K}_+ \times \mathcal{K}_+) 
    \quad \text{ and } \quad 
    S_- \vcentcolon = A \cap (\mathcal{K}_- \times \mathcal{K}_-) 
\end{align*}
are closed simple symmetric operators with deficiency index $(1,1)$.  Conversely, we   show that a realization involving a relation with such a decomposition property corresponds to a function of bounded type.  Relations satisfying such a decomposition property are called partially fundamentally reducible and have been extensively studied in the past (see e.g. \cite{article36278}, \cite{article43879} and \cite{article55}). 
It is important to note that these relations are relatively ``close'' to self-adjoint relations on Hilbert spaces. More precisely, let $A$ be a partially fundamentally reducible relation in a \emph{Krein space} $\mathcal{K}=\mathcal{K}_+ [+] \mathcal{K}_-$.  Then there exists a self-adjoint relation $A_0$ in the \emph{Hilbert space} $ \mathcal{K}_+ \oplus (- \mathcal{K}_-)$ such that the resolvent of $A$ is  a one-dimensional perturbation of the resolvent of $A_0$ (see \cite[Theorem 4.7]{article55}). 
\color{black}
\par\smallskip
This article is structured as follows. Since we use quite advanced results from various fields of complex analysis and operator theory, we decided to include an extensive preliminary section, that is Section 2. While the content of this section isn't original in a strict sense, it is important for the understanding of this article as we present some well-known concepts in a new light.  
In Section 3, we construct minimal realizations based on the representation given by \eqref{Helsonsrep}, with its proof to be presented in Section 4. 
Moving forward, in Section 5 we study  the constructed relations and in Section 6 we discuss some illustrative examples. 
\color{black}
\par\smallskip
Finally, we mention that the idea of using suitable reproducing kernel function spaces as models for (generalized) Herglotz-Nevanlinna functions is quite old, see for example \cite{DeBranges}, \cite{MR1318517}, 
\cite{MR1115445}, \cite{ea00d92c00704a44bedf537e4048aecf}, and also \cite{edsjsr.119452920031201}. 
\color{black}
In addition, we also mention that the ideas developed in \cite{Alpay1} had a great influence on this article.
There the author studied a closely related question in the context of inverse scattering and gave some partial answers. The methods developed in this article allow us to solve that problem in full generality as well (see Section \ref{GenBhspaces} for details).

\section{Preliminaries}
\subsection{Krein spaces and operator models}
This section aims to formally introduce operator models. We begin by discussing fundamental concepts from the theory of Krein spaces, starting with their definition \cite[Section 1.4]{gheondea_2022}: 
\begin{definition}
    Let $\mathcal K$ be a complex vector space and $[\cdot,\cdot]_{\mathcal{K}}$ a sesquilinear form on $\mathcal K$. Then $\mathcal K$ is a Krein space if $\mathcal{K}$ is decomposable into a direct orthogonal sum 
\begin{align*}
\mathcal{K} = \mathcal{K}_+ [+] \mathcal{K}_-,
\end{align*} 
such that $(\mathcal{K}_+, [\cdot,\cdot]_{\mathcal{K}})$ and $-\mathcal{K}_-\vcentcolon = (\mathcal{K}_-, -[\cdot,\cdot]_{\mathcal{K}})$ are Hilbert spaces. 
\end{definition}
We equip a Krein space $\mathcal K$ with the Hilbert space topology of $\mathcal{K}_+ \oplus (-\mathcal{K}_-)$, which is independent of the choice of the decomposition (see e.g. \cite[Section 1.4]{gheondea_2022}). The  space of continuous linear operators on $\mathcal K$ is  then denoted  by $\mathscr{L} (\mathcal{K})$. 
\par\smallskip
Realizations are closely related to self-adjoint relations on Krein spaces.  
Here, a relation  is a linear subspace $A \subset \mathcal{K} \times \mathcal{K}$, its adjoint $A^+$ is defined as
\begin{align*}
    A^+=\{(u,v) \in \mathcal{K}\times \mathcal{K}\vcentcolon \forall (x,y) \in A \vcentcolon [u,y]_{\mathcal{K}}=[v,x]_{\mathcal{K}} \},
\end{align*}
and $A$ is self-adjoint  if $A=A^+$. We refer to \cite{DijksmaDeSnoo87} for a comprehensive introduction to the theory of self-adjoint relations in the indefinite setting. 
Finally,
realizations are  defined as follows \cite{article}:
\begin{definition}\label{Def:PAN1}
Let  $q \vcentcolon  U \subset \C  \rightarrow \C $ be an analytic function. Then $q$ admits a realization on a Krein space $\mathcal K$ if there exists 
\begin{itemize}
    \item a self-adjoint relation $A$ in $\mathcal K$ with non-empty resolvent set
    \item a base point $\zeta_0 \in \varrho(A)\cap U $ and a generating vector $v \in \mathcal{K}$
\end{itemize}
such that 
\begin{align}\label{For:PAN1}
q(\zeta)= \overline{q(\zeta_0)}+(\zeta- \overline{\zeta_0}) \left[\left(I+(\zeta-\zeta_0)(A-\zeta)^{-1}\right) v, v\right]_{\mathcal K} \quad \forall \zeta \in \varrho(A)\cap U.
\end{align}
In this case, the triple $(\mathcal{K}, A, v)$ is called a realization of $q$. The realization is called minimal, if 
\begin{align*}
    \overline{\mathrm{span}} \{ (I+(\zeta-\zeta_0)(A-\zeta)^{-1}) v \vcentcolon \zeta \in \varrho(A) \}= \mathcal K.
\end{align*}
A minimal realization is also called operator model. 
\end{definition}
Since $A$ is self-adjoint,  the resolvent set $\varrho(A)$ is symmetric with respect to $\R$, and $q$ satisfies $q(\overline{\zeta})=\overline{q(\zeta)}$ whenever both $\zeta$ and $\overline{\zeta}$ are in $ \varrho(A)\cap U$. Thus,  if $q$ is defined on a subset $U \subset \C^+$, then a realization prescribes an extension to a subset of the lower half plane $\C^-$.
\par\smallskip
Every 
realization induces a so called
\textit{defect function}, which is defined as follows (see e.g. \cite{2015ot}): 
\begin{definition} \label{Def:defect}
    Let $q \vcentcolon  U \subset \C  \rightarrow \C$ be an analytic function and $(\mathcal{K},A,v)$ a realization.  The  defect function $\phi$ is defined as 
\begin{align*}
  \phi \vcentcolon \varrho(A) \rightarrow \mathcal{K}, \quad 
    \phi(\zeta)= \left(I+(\zeta-\zeta_0)(A-\zeta)^{-1}\right) v. 
\end{align*}
    Using the resolvent identity, it can be shown that for all $\zeta, w \in \varrho(A)$ it holds that
\begin{align}\label{For:DefectFunction}
\begin{split}
    &\phi(\zeta)= (I+(\zeta-w)(A-\zeta)^{-1})\phi(w)
     \\ \text{ and} \quad  
&[\phi(\zeta),\phi(w)]_{\mathcal{K}}=\frac{q(\zeta)-\overline{q(w)}}{\zeta-\overline{w}}, \  \zeta \neq \overline w , \quad
[\phi(w),\phi(\overline{w})]_{\mathcal{K}}=q'(w), 
\end{split}
\end{align}
where we extend $q$ to the whole of $\varrho(A)$ by the right hand side of \eqref{For:PAN1} if necessary.
Note that \eqref{For:DefectFunction} means that $(\mathcal{K},A,\phi(w))$ is a realization of $q$ with base point $w$. We occasionally write $(\mathcal{K},A,\phi)$ instead of $(\mathcal{K},A,v)$ if we want to suppress the dependence on the base point $\zeta_0$.
\end{definition}
Finally, two realizations $(\mathcal{K}_i,A_i,v_i), \ i \in \{1,2\}$ of the same function $q$ are isomorphic if there exists a bounded unitary operator $F \vcentcolon  \mathcal{K}_1 \rightarrow \mathcal{K}_2$ such that
$F A_1 F^{-1} =A_2$ and $F(\phi_1(\zeta))=\phi_2(\zeta)$ for all $\zeta \in \varrho(A_1)=\varrho(A_2)$.

\subsection{Functions of bounded type} \label{Sec:BoundedType}
This section focuses on meromorphic functions of bounded type and their fundamental properties. In order to discuss them in a meaningful way, we begin by reviewing the Hardy spaces  $\mathcal{H}^2(\C^+)$ and $\mathcal{H}^\infty(\C^+)$ along with inner and outer functions. A comprehensive introduction to their theory is given in \cite[Section 5]{Rosenblum1994tih}, where  all of the upcoming definitions and results can be found. 
\par\smallskip
Throughout this article, 
we denote by $\lambda$ the Lebesgue measure on $\R$ and by $\mathcal{O}(U)$ the space of analytic functions defined on an open set $U\subset \C$. The Hardy spaces  $\mathcal{H}^2(\C^+)$ and $\mathcal{H}^\infty(\C^+)$ are then defined as
\begin{align*}
\mathcal{H}^2(\C^+) \vcentcolon  &=  \left\{f  \in \mathcal{O}(\C^+) \vcentcolon \|f\|_{\mathcal{H}^2(\C^+)} \vcentcolon = 
\sup_{0<y < \infty}\left(\int_{\R} |f(x+iy)|^2 \mathrm{d} \lambda(x) \right) ^{\frac{1}{2}} < \infty\right\}
 \\
\mathcal{H}^\infty(\C^+) \vcentcolon &=  \left\{f \in \mathcal{O}(\C^+) \vcentcolon \|f\|_{\mathcal{H}^\infty(\C^+)} \vcentcolon = \sup_{\zeta\in \C^+} |f(\zeta)| < \infty\right\}.
\end{align*}
Both $\mathcal{H}^2(\C^+)$ and  $\mathcal{H}^\infty(\C^+)$ are complete. 
Another important property  is the existence of boundary values. More precisely, for any given  $f \in \mathcal{H}^2(\C^+)$ ($f \in  \mathcal{H}^\infty(\C^+)$)   the limit 
\begin{align*}
    f(x) \vcentcolon =\lim_{\epsilon \downarrow 0} f(x+\mathrm{i}\epsilon) \text{ exists $\lambda$-a.e. }
\end{align*}
In fact, the same is true even for non-tangential boundary limits, but this will not be used in this article. 
Either way, this allows us to identify a function with its boundary function, which   defines an isometric embedding $\mathcal{H}^2(\C^+) \hookrightarrow L^2(\R,\lambda)$. Consequently,  $\mathcal{H}^2(\C^+)$ is even a Hilbert space. 
\par\smallskip
The elementary building blocks of functions in $\mathcal{H}^\infty(\C^+)$ are  inner and outer functions. A function $V\in \mathcal{H}^\infty(\C^+)$ is called inner, if it satisfies $\|V\|_{\mathcal{H}^\infty(\C^+)}\leq 1$ and $|V(\zeta)|=1$ $\lambda$-a.e. on $\R$.  An outer function is a function of the form 
\begin{align*}
F(\zeta) \vcentcolon = \exp \left( \frac{-\mathrm{i}}{\pi} \int_\R \left( \frac{1}{t-\zeta}-\frac{t}{1+t^2} \right)\log(|h(t)|)\mathrm{d}\lambda(t) \right),
\end{align*}
where $h$ is a function on $\R$ satisfying  $h \geq 0 \ \lambda$-a.e. and 
\begin{align*}
    \int_\R  \frac{h(t) \mathrm{d}\lambda(t)}{1+t^2}<\infty.
\end{align*}
It can be shown that $h$ is given by  $\log(|h(t)|)=|F(t)|$. Finally, for any given function $f \in \mathcal{H}^\infty(\C^+)$, there exists an inner function $V$ and an outer function $F$ such that  
\begin{align*}
    f = V \cdot F.
\end{align*}
This concludes our brief discussion of Hardy spaces, and we are ready to define meromorphic functions of bounded type and the Smirnov class: 
\begin{definition}
Let $U \subset \C^+$ be such that $\C^+ \setminus U$ consists of isolated points and let $f \vcentcolon U\subset \C^+ \rightarrow \C$ be a meromorphic function on $\C^+$. Then $f$ is of bounded type if there are two functions
$f_1, f_2 \in \mathcal{H}^\infty(\C^+)$ 
 such that 
\begin{align*}
f (\zeta) =\frac{f_1(\zeta)}{f_2(\zeta)} \quad \text{ for all } \zeta \in U. 
\end{align*}
We denote this class by $\mathcal{N}(\C^+)$. 
\par\smallskip
Moreover, if $f_2$ can be chosen to be an outer function, then $f$ is in the Smirnov class which we denote by $\mathcal{N}^+(\C^+)$. Note that in this case $f$ is analytic on $\C^+$, since an outer function does not have any zeros. 
\end{definition}
Note that $f_1$ and $f_2$ can be chosen to satisfy $\|f_1\|_{\mathcal{H}^\infty(\C^+)} \leq 1$ and $\|f_2\|_{\mathcal{H}^\infty(\C^+)} \leq 1$ by scaling them appropriately. 
Functions of bounded type inherit many properties from bounded analytic functions, for example the existence of boundary values.
\par\smallskip
Throughout this article, we will often encounter functions whose boundary functions are elements of  $L^2(\R)$. In that situation,  the following regularity result for the Smirnov class will be useful \cite[Theorem 5.23]{Rosenblum1994tih}: 
\begin{proposition}\label{Prop:Regres}
    Let $f \in \mathcal{N}^+(\C^+)$ and assume that its boundary function is an element of $L^2(\R)$. Then $f \in \mathcal{H}^2(\C^+)$.
\end{proposition}
We emphasize that the a priori assumption $f \in \mathcal{N}^+(\C^+)$ is essential and cannot be weakened. 
\par\smallskip
Functions of bounded type and the Smirnov class on the lower-half plane $\C^-$ are defined  analogously and denoted by $\mathcal{N}(\C^-)$ and $\mathcal{N}^+(\C^-)$ respectively. The spaces $\mathcal{N}(\C^+)$ and $\mathcal{N}(\C^-)$  are connected via the Schwartz reflection, which is the vector space isomorphism given by
\begin{align*}
    SR \vcentcolon \mathcal{N}(\C^+) \rightarrow \mathcal{N}(\C^-), \quad SR(f)(z)=\overline{f(\overline{z})}.
\end{align*}
Note that $SR$ maps $\mathcal{H}^2(\C^+)$ and $\mathcal{N}^+(\C^+)$ onto  $\mathcal{H}^2(\C^-)$ and $\mathcal{N}^+(\C^-)$ respectively, and preserves inner and outer functions. 
\par\smallskip
Now recall that a realization of a function $q$ defined on $\C^+$ prescribes an extension to $\C^-$ via  $q(\overline{\zeta})=\overline{q(\zeta)}$. In order to take this extension into account, we define the following function classes:
\begin{align*}
    \mathcal{N}(\C \setminus \R) \vcentcolon &= \{ f \text{ meromorphic on } \C \setminus \R \vcentcolon f_{|\C^+} \in \mathcal{N}(\C^+) \text{ and } 
    f_{|\C^-} \in \mathcal{N}(\C^-)\} \\
    \mathcal{N}^+(\C \setminus \R) \vcentcolon &= \{f \in \mathcal{N}(\C \setminus \R) \vcentcolon f_{|\C^+} \in \mathcal{N}^+(\C^+) \text{ and } f_{|\C^-} \in \mathcal{N}^+(\C^-)   \} \\
    \mathcal{N}_{sym} \vcentcolon &= \{f \in \mathcal{N}(\C \setminus \R) \vcentcolon f_{|\C^-} = SR(f_{|\C^+})   \}.
\end{align*} 
Taking the lower-half plane $\C^-$ into account may not seem  important at  first glance. However, this rather subtle difference turns out to be significant for reasons that will become apparent in Section \ref{Sec:Bh spaces}.  
\par\smallskip
We will often consider the Cayley transform of  functions in  $\mathcal{N}_{sym} $, which is defined as 
\begin{align}\label{For:Cayley}
    C \vcentcolon \mathcal{N}_{sym} \setminus \{ f \in  \mathcal{N}_{sym} \vcentcolon f_{|\C^+}\equiv \mathrm{i} \text{ or } f_{|\C^+}\equiv -\mathrm{i} \} \rightarrow \mathcal{N}(\C\setminus \R), \quad f \mapsto C(f)= \frac{1+\mathrm{i}f}{1-\mathrm{i}f}.
\end{align}
We collect some properties in the next Proposition.
\begin{proposition}\label{Prop:Cayley}
The Cayley transform $C$ is injective and its inverse is given by
\begin{align}\label{For824039}
    f=C^{-1}(C(f))= \mathrm{i}  \frac{1- C(f)}{1+C(f)}.
\end{align}
Moreover, the symmetry relation 
$f_{|\C^-}=SR(f_{|\C^+})$ transforms to $C(f)_{|\C^-}=\frac{1}{SR(C(f)_{|\C^+})}$.
\end{proposition}
\begin{proof}
It is straightforward to check that the inverse of $C$ is indeed given by \eqref{For824039}. Moreover, we calculate for $ \zeta \in \C^-$ 
\begin{align*}
     C(f)_{|\C^-}(\zeta)
      =\frac{1+\mathrm{i} f_{|\C^-}(\zeta)}{1-\mathrm{i} f_{|\C^-}(\zeta)}
     = \frac{1+\mathrm{i}  \overline{f_{|\C^+}(\overline{\zeta})}}{1-\mathrm{i} \overline{f_{|\C^+}(\overline{\zeta})}} 
     = \overline{\left(\frac{1-\mathrm{i}  f_{|\C^+}(\overline{\zeta})}{1+\mathrm{i} f_{|\C^+}(\overline{\zeta})}\right)}=\overline{\left(\frac{1}{C(f)_{|\C^+}(\overline{\zeta})}\right)},
\end{align*}
where we have used   $f_{|\C^-}(\zeta)=\overline{f_{|\C^+}(\overline{\zeta})}$ in the second step. This concludes the proof. 
\end{proof}
\color{black}

\par\smallskip
Finally, let $q$ be a generalized Nevanlinna function and $(\mathcal{K},A,v)$ be a minimal realization. Then it is well known \cite{2015ot}, that $q$ is analytically continuable through an open interval $(a,b)\subset \R$ if and only if $(a,b)\subset \varrho(A)$. This characterizes the resolvent set of $A$ completely in terms of the function theoretic properties of $q$. We want to prove an analogous result for functions in $\mathcal{N}_{sym}$. To this end, we give the following definition:
\begin{definition}
    Let $f \in  \mathcal{N}(\C \setminus \R)$ and $U(f)$ its domain. We define $U_{ext}(f)$ as the union of $U(f)$ and all $\zeta \in \R$ for which there exists an open interval  $\zeta\in (a,b)\subset \R$ such that $f$ can be analytically extended through $(a,b)$.
\end{definition}

\subsubsection{Calculus of inner functions}
Inner functions have attracted significant interest in recent decades, and their theory plays a crucial role in our construction of operator models. This is why we discuss them in detail now. As before, all  results and definitions can be found in \cite{Rosenblum1994tih} and \cite{2019hs}.
\color{black}
\par\smallskip
First, we recall that every inner function $V$ can be  decomposed into two conceptually different inner functions, namely into a Blaschke product $B$ and a singular part $S$ \cite[Theorem 5.13]{Rosenblum1994tih}. If $(\zeta_n)$ denotes the zeros of $V$ (counting multiplicities), then the Blaschke product $B$ is defined as  
\begin{align*}
B(\zeta) &\vcentcolon = \prod_{n} \exp(\mathrm{i}  \alpha_n)  \frac{\zeta-\zeta_n}{\zeta-\overline{\zeta_n}},
\ \text{ where } \
\alpha_n \in \R \text{ satisfies } \exp(\mathrm{i}  \alpha_n)  \frac{\mathrm{i}-\zeta_n}{\mathrm{i}-\overline{\zeta_n}} \geq 0.
\end{align*}
 Note that $B$ has the same zeros as $V$. Then $V$ can be written as 
\begin{align*}
V(\zeta) = \lambda  B(\zeta)  \exp \left( \frac{\mathrm{i}}{\pi} \int_\R \left( \frac{1}{t-\zeta}-\frac{t}{1+t^2} \right)\mathrm{d}\nu(t) \right)  \exp(\mathrm{i} \alpha \zeta),
\end{align*}
where $\lambda \in \T$, $\nu$ is a singular Borel measure satisfying $\int_\R \frac{\mathrm{d}\nu(t)}{1+t^2}<\infty$ and $\alpha \in \R_{\geq 0}$. An inner function given by such an exponential representation is called singular. 
\par\smallskip
The class of inner functions admits a multiplicative structure, which is defined as follows \cite[Definition 2.16]{garcia2016introduction}:
\color{black}
\begin{definition}
Let $V_1$ and $V_2$ be inner functions. Then we say that 
\begin{itemize}
    \item the function $V_1$ divides $V_2$ if there exists an inner function $V_3$ such that $V_2=V_1 \cdot V_3$.
    \item the functions  $V_1$ and $V_2$ are relatively prime,  if there exists no non-constant inner function which divides both $V_1$ and $V_2$.
\end{itemize}
\end{definition}
It is possible to characterize when $V_1$ and $V_2$ are relatively prime in terms of the zeros of the Blaschke product and the measure of the singular part \cite[Section 3.2.1]{2019hs}: 
\begin{proposition}
    Let $V_j=\lambda_j  B_j \cdot S_j \cdot \exp(\mathrm{i} \alpha_j \zeta) , \ j \in 
    \{1,2\}$ be two inner functions, where $\lambda_j \in \T$, $B_j$ is the Blaschke product and $S_j$ the singular part. Then $V_1$ and $V_2$ are relatively prime if  and only if 
    \begin{itemize}
        \item the Blaschke products have no common zeros,
        \item either $\alpha_1$ or $\alpha_2$ is zero,
        \item and the measures $\mu_1$ and $\mu_2$ representing $S_1$ and $S_2$ are mutually singular.
    \end{itemize}
\end{proposition}
Finally, we  discuss the boundary behaviour of inner functions. To this end, the notion of a pseudocontinuation is useful 
\cite[Section 2]{Douglas1970}:
\begin{definition}
Let $f\in \mathcal{N}(\C^+)$ and $g\in \mathcal{N}(\C^-)$. Then $f$ and $g$ are said to be pseudocontinuations of each other if their boundary functions coincide $\lambda$-a.e.
\end{definition}
Given an inner function $V$, it is straightforward  to verify that  $\frac{1}{SR(V)}$ is a pseudocontinuation of $V$. The following important regularity result is an immediate consequence of Proposition \ref{Prop:Regres}.
\begin{proposition}\label{Prop:RegRes}
    Let $f \in \mathcal{H}^2(\C^+), g \in \mathcal{N}(\C^-)$ and assume that $f$ and $g$ are pseudocontinuations of each other.  Then either $f\equiv 0$ and $ g \equiv 0$ or $ g \notin \mathcal{N}^+(\C^-)$. 
\end{proposition}
\color{black}
\begin{proof}
The boundary function of $g$ coincides by assumption with the one of  $f$ and is therefore in  $L^2(\R)$. Consequently, if we assume that $ g \in \mathcal{N}^+(\C^-)$  then $g \in \mathcal{H}^2(\C^-)$ by Proposition \ref{Prop:Regres}. In this case,  also the Schwartz reflection $SR(g)$ is an element of  $\mathcal{H}^2(\C^+)$. Moreover,  it holds that
\begin{align*}
    SR(g)(x) = \overline{g(\overline{x})}=\overline{g(x)}= \overline{f(x)} \quad \text{ for  $\lambda$ almost all $x$ in  $\R$. }
\end{align*}
    However,  there is only one function  $h \in \mathcal{H}^2(\C^+)$ such that there exists a function $m \in \mathcal{H}^2(\C^+)$ satisfying 
    $h(x)=\overline{m(x)} for  \ \lambda$-a.e. $x \in \R$, and that is  the constant zero function \cite[Page 112]{Rosenblum1994tih}.
    Therefore, we arrive at  $f\equiv 0$ and $SR(g)\equiv 0$, which concludes the proof. 
\end{proof}
Finally, we can also characterize when the pseudocontinuation is even an analytic continuation \cite[Theorem 3.2.3]{2019hs}: 
\begin{proposition}\label{Prop:AnaExt}
    Let $V= \lambda  B \cdot S \cdot \exp(\mathrm{i} \alpha \zeta)$ be an inner function and $\mu_s$ the measure representing the singular part. Then $V$ can be analytically continued through an interval $(a,b)\subset \R$ if and only if 
    \begin{itemize}
    \item
        the zeros of the Blaschke product do not have an accumulation point in $(a,b)$
    \item 
        the interval $(a,b)$ does not intersect the topological support of $\mu_{s}$. 
    \end{itemize}
    In this case, the analytic continuation coincides with the pseudocontinuation. 
\end{proposition}
Recall that the topological support of a positive measure $\mu$ is defined as the set
\begin{align*}
    \{x \in \R \vcentcolon \mu(O)>0 \text{ for every open neighbourhood $O$ of $x$} \}.
\end{align*}

\subsection{Modeling reproducing kernel spaces}
Minimal realizations are  closely connected to a specific class of reproducing kernel Krein spaces, which we will call ``modeling''. In what follows, we first discuss the definition and basic properties of reproducing kernel spaces in the indefinite setting, and then explain this connection in detail. A comprehensive introduction to reproducing kernel Krein spaces is given in \cite{2015otb}.
\begin{definition}\label{Def:RepKernelSpaces}
    Let $U \subset \C$ be an open set and $\mathcal K$ a Krein space consisting of analytic functions defined on $U$. Then $\mathcal K$ is a reproducing kernel Krein space on $U$ if there exists a function $ H \vcentcolon U \times U \rightarrow \C$
such that 
\begin{enumerate}
    \item for every $w \in U$ the function $\zeta \mapsto  H(\zeta,w)$ belongs to $\mathcal K$.
    \item for every $f \in \mathcal K$ and $w \in U$ it holds that $f(w)=[f(\cdot),H(\cdot,w)]_{\mathcal{K}}$.
\end{enumerate}
\end{definition}
We call the functions $H(\cdot,w)\in \mathcal{K}$ for $w \in U$ kernel functions. Note that they form a total set, i.e. it holds that 
\begin{align*}
    \overline{\mathrm{span}} \{H(\cdot,w) \vcentcolon w \in U  \}= \mathcal{K}.
\end{align*}
Indeed, any function $f \in \mathrm{span} 
\{H(\cdot,w) \vcentcolon w \in U  \}^\perp$ satisfies 
\begin{align*}
    f(w)=[f(\cdot),H(\cdot,w)]_{\mathcal{K}}=0 \quad \forall w \in U.
\end{align*}
Similar to the Hilbert space case, various operations can be applied to reproducing kernel Krein spaces
\cite[Section 5]{2016ait}. For this article, we focus specifically on two operations: sums and products. 
\begin{proposition}[Summation]
Let $\mathcal{K}_1$ and $\mathcal{K}_2$ be two reproducing kernel Krein spaces on $U$ and $H_1$ and $H_2$ their respective kernels. If it holds that
\begin{align*}
\mathcal{K}_1 \cap \mathcal{K}_2 =\{0\},
\end{align*}
then $\mathcal{K}_1 \oplus \mathcal{K}_2$ is a reproducing kernel Krein space with kernel 
$H_1+H_2$.
\end{proposition}
\begin{proof}
    Since $\mathcal{K}_1 \cap \mathcal{K}_2 =\{0\}$, we can interpret elements in the direct sum $\mathcal{K}_1 \oplus \mathcal{K}_2  $ as analytic functions instead of  pairs of analytic functions. It is then straightforward to check that
    \begin{align*}
f(w)=
[f(\cdot),H_1(\cdot,w)+H_2(\cdot,w)]_{\mathcal{K}_1 \oplus \mathcal{K}_2}
    \end{align*}
    for $f \in \mathcal{K}_1 \oplus \mathcal{K}_2$ and $w \in U$. 
\end{proof}
\begin{proposition}[Products]\label{Prop:Conj}
Let $\mathcal{K}$ be a reproducing kernel Krein space on $U$ and kernel $H$. Moreover, let
$g\in \mathcal{O}(U)$ be an analytic function which is not identically zero on any component of $U$. 
Consider the vector space $g \cdot \mathcal{K} \subset \mathcal{O}(U)$ and the operator 
\begin{align*}
    M_g \vcentcolon \mathcal{K} \rightarrow g \cdot \mathcal{K}, \quad  f \mapsto g \cdot f,
\end{align*}
which is bijective by our assumption on $g$. We equip  $g \cdot \mathcal{K}$ with the inner product that turns  $M_g$ into an isomorphism of Krein spaces. Then $g \cdot \mathcal{K}$
is a reproducing kernel Krein space on $U$ with kernel 
\begin{align*}
    H_g(\zeta,w) \vcentcolon = g(\zeta)  H(\zeta,w)  \overline{g(w)}.
\end{align*}
\end{proposition}
\begin{proof}
    The usual proof for the Pontryagin space case, which can be found in \cite[Theorem 1.5.7]{Alpay1997}, works exactly the same here.
\end{proof}
Finally, we  discuss a common way to deal with domain issues. Let $\mathcal{K}$ be a reproducing kernel Krein space on $U$ and $E$ a discrete set in $U$. Then we can ``forget'' about the points in the set $E$ by restricting the functions in $\mathcal{K}$ to the set $U \setminus E$. This is obviously a bijective operation, which means that we end up with an isomorphic reproducing kernel Krein space $\tilde{\mathcal{K}}$, whose kernel is the restriction of the original kernel $H$ to  $U \setminus E \times U \setminus E$. This is  especially convenient, when we want to consider multiplications of $\mathcal{K}$ by meromorphic functions on $U$. In this case, we just exclude the poles of said function from our original domain. 

\subsubsection{Modeling reproducing kernel spaces}
Operator models are closely related to reproducing kernel Krein spaces generated by the Nevanlinna kernel:
\begin{definition}
Let $f \in \mathcal{N}_{sym}$ and $U_{ext}(f)$ its extended domain.
The Nevanlinna kernel $N_f$ is then defined as
\begin{align*} 
N_f\vcentcolon U_{ext}(f) \times U_{ext}(f) \rightarrow \C, \quad 
N_f(\zeta,w)
:=\frac{f(\zeta)-\overline{f(w)}}{\zeta-\overline w}, \  \zeta \neq \overline w, \quad
N_f(w,\overline{w})=f'(w). 
\end{align*}
\end{definition}
Note that $N_f$ has already made an appearance in  Definition \ref{Def:defect}. The Nevanlinna kernel is closely related to the  difference-quotient operator which is defined as follows:
\begin{definition}\label{Def:diffquot}
Let $U \subset \C$ be an open set and $w \in U$.  The difference-quotient  operator $D_w$ is defined as
\begin{align*}
   D_w \vcentcolon \mathcal{O}(U) \rightarrow \mathcal{O}(U), \quad
     D_w(f)(\zeta) = \frac{f(\zeta)-f(w)}{\zeta-w},  \  \zeta \neq  w, \quad
D_w(f)(w)=f'(w). 
\end{align*}
It is straightforward to verify that the difference-quotient operator satisfies the following two identities:
\begin{itemize}
\item[(i)] The resolvent identity: $D_\lambda-D_\mu= (\mu-\lambda) D_\lambda D_\mu$, where $\lambda, \mu \in  U$.
\item[(ii)] The multiplication identity: $D_\lambda(f \cdot g)= f(\lambda)  D_\lambda(g) + g \cdot D_\lambda(f)  $, where $f,g \in  \mathcal{O}(U)$.
\end{itemize}
\end{definition}
These two concepts allow us to define the analog of a minimal realization in the realm of reproducing kernels:
\begin{definition}
    Let $f \in \mathcal{N}_{sym}$ and $U_{ext}(f)$ its extended domain. Moreover, let  $\mathcal{K}$ be a reproducing kernel Krein space on $\Omega \subset U_{ext}(f)$  with kernel $N_f$ , which is symmetric with respect to $\R$. Then $\mathcal{K}$ is called a modeling reproducing kernel space on $\Omega$ for $f$   if  
    \begin{align*}
        D_w \vcentcolon \mathcal{K} \rightarrow \mathcal{K}
    \end{align*}
    is well-defined and bounded for  every $w \in \Omega$. 
\end{definition}
Every modeling reproducing kernel space gives rise to a minimal realization of a particular form, which is  called canonical in the context of generalized Nevanlinna functions \cite{edsjsr.119452920031201}:
\begin{proposition}\label{Prop:Canreal}
    Let $f \in \mathcal{N}_{sym}$ and $\mathcal{K}$ be a modeling reproducing kernel space on $\Omega$  for $f$. Then there exists a self-adjoint relation $A$ such that $(A-w)^{-1}=D_w$ and $\Omega \subset \varrho(A)$. Moreover, 
    for fixed $w \in \Omega$ we can express $A$ as
    \begin{align}
    \begin{split}\label{FormA}
        A&=\{ (D_w(g),g+w D_w(g)) \in \mathcal{K} \times \mathcal{K}, \ g \in  \mathcal{K}   \} \\
        &= \{ (h,g)\in \mathcal{K} \times \mathcal{K} \vcentcolon \exists c \in \C \vcentcolon g(\zeta)-\zeta  h(\zeta)\equiv c \}.
    \end{split}
    \end{align}
    \color{black}
    Finally, the triple $(\mathcal{K},A,N_f(\cdot,\overline{\zeta_0}))$ is a minimal realization of $f$  with base point $\zeta_0\in \Omega$. 
    \end{proposition}
    \begin{proof}
        Consider the operator valued 
        function 
        \begin{align*}
            T \vcentcolon \Omega \rightarrow \mathscr{L}(\mathcal{K}), \quad 
            w \mapsto D_w.
        \end{align*}
        The resolvent identity from Definition \ref{Def:diffquot} can be reformulated as
        \begin{align*}
            T(\lambda)-T(\mu)= (\lambda-\mu) T(\lambda) T(\mu).
        \end{align*}
        Moreover, a short calculation using the reproducing kernel property shows that 
       \begin{align*}
        [D_w (N_f(\cdot,z)),  N_f(\cdot,\zeta)]_{\mathcal{K}} &= 
        \left[N_f(\cdot,z),D_{\overline{w}} (N_f(\cdot,\zeta))\right]_{\mathcal{K}} \quad \forall z,\zeta \in \Omega.
    \end{align*} 
     Since the kernels are dense in $\mathcal{K}$, we conclude that $T(w)=T(\overline{w})^+$. Consequently, it follows that there exists a self-adjoint relation $A$ such that $\Omega \subset \varrho(A)$ and 
    \begin{align*}
         (A-w)^{-1}=T(w)=D_w,
    \end{align*}
    see \cite[Proposition 1.1.]{RN07186998219990101}. Finally, it is straightforward to verify that $A$ is of the form described in equation \eqref{FormA}.
    \par\smallskip
    Now pick any point $\zeta_0 \in \Omega$ and calculate 
    \begin{align}\label{For:Minimality}
        (I+(w-\zeta_0)(A-w)^{-1}) N_f(\cdot,\overline{\zeta_0})= N_f(\cdot,\overline{w}). 
    \end{align}
    Consequently, we conclude using the reproducing kernel property and $f(\overline{\zeta})=\overline{f(\zeta)}$ that 
    \begin{align*}
    [(I+(\zeta-\zeta_0)(A-\zeta)^{-1}) N_f(\cdot,\overline{\zeta_0}) , N_f(\cdot,\overline{\zeta_0})]_{\mathcal{K}(f)}
    &=[N_f(\cdot,\overline{\zeta}) , N_f(\cdot,\overline{\zeta_0})]_{\mathcal{K}(f)} \\
    &=N_f(\overline{\zeta_0},\overline{\zeta})
    = \frac{f(\zeta)-\overline{f(\zeta_0)}}{\zeta-\overline{\zeta_0}}.
    \end{align*}
    This means that $(\mathcal{K},A,N_f(\cdot,\overline{\zeta_0}))$ is indeed a realization of $f$  with base point $\zeta_0$. Finally, the kernel functions form a total set in $\mathcal{K}$, i.e. the following holds (see the remark after Definition \ref{Def:RepKernelSpaces}):
    \begin{align*}
    \overline{\mathrm{span}} \{N_f(\cdot,w) \vcentcolon w \in  \Omega \}= \mathcal{K}
    \end{align*}
    Therefore, equation \eqref{For:Minimality} ensures that this realization is also minimal. 
    \end{proof}
    \color{black}
Thus any modeling reproducing kernel space is associated to a minimal  realization.  Conversely, any minimal realization is also associated to a modeling reproducing kernel space \cite[Theorem 6.4]{me2}:
\begin{proposition}\label{Prop:RKSIso}
    Let $(\mathcal{K},A,v)$ be a minimal realization of $f$ with base point $\zeta_0 \in \varrho(A)$. Then $\mathcal{K}$ is isomorphic to a modeling reproducing kernel space $\tilde{\mathcal{K}}$ on  $\varrho(A)$. The isomorphism is given by 
    \begin{align*}
        F\vcentcolon \mathcal{K} \rightarrow \tilde{\mathcal{K}}, \quad x \mapsto F(x), \ F(x)(\zeta) \vcentcolon =[x,\phi(\overline{\zeta})]_{\mathcal{K}},
    \end{align*}
    where $\phi$ is the defect function of $(\mathcal{K},A,v)$.
\end{proposition} 
This concludes our reformulation of (minimal) realizations in terms of reproducing kernels. Working with modeling reproducing kernel spaces is advantageous as their definition does not involve linear relations.
Finally,  some words on the uniqueness of operator models are in order:
\begin{remark}
    Modeling reproducing kernel spaces, and in turn minimal realizations, are in general only weakly unique. Indeed, let  $\mathcal{K}_1(f)$ and $\mathcal{K}_2(f)$ be two such spaces which are defined on the same set $U$.  Then the identity operator 
    \begin{align*}
        \mathrm{Id} \vcentcolon \mathrm{span}
    \{ N_f(\cdot,w) \vcentcolon  w \in  U\}
    \subset
    \mathcal{K}_1(f) \rightarrow \mathcal{K}_2(f), \quad 
     N_f(\cdot,w_i) \mapsto
    N_f(\cdot,w_i),
    \end{align*}
    is isometric, but in general not bounded. An example illustrating this point can be found in \cite[Section 2]{Jonas}. However, it is important to note that minimal realizations are unique up to a bounded unitary operator if $\mathcal{K}$ is a Pontryagin space \cite[Remark 1.7]{article}. 
    \end{remark}

\subsection{Herglotz-Nevanlinna functions and  realizations on Hilbert spaces} \label{Sec:HN functions}
The goal of this section is to discuss  Herglotz-Nevanlinna functions and the corresponding operator models on Hilbert spaces, which will play a crucial role in  our construction. Their definition can be formulated in  multiple equivalent ways, as detailed  in \cite{edssjs.7395739519561201} and the survey \cite{LugerOu}: 
\begin{proposition}\label{Prop:Herglotz}
Let $q \vcentcolon \C\setminus \R \rightarrow \C$ be an analytic function satisfying $q(\overline{\zeta})=\overline{q(\zeta)}$. Then the following are equivalent:  
\begin{itemize}
\item[(A)] The function $q$ is a Herglotz-Nevanlinna function, i.e., ${\rm Im } \,q(\zeta)\geq 0$ for ${\rm Im }\,\zeta> 0$.
\item[(B)]
There exist numbers $a\in\mathbb R$ and $b\geq0$ and a positive Borel measure $\nu$ satisfying
\begin{align*}
    \int_\R \frac{\mathrm{d}\nu(t)}{1+t^2}<\infty
\end{align*}
such that
\begin{equation}\label{eq:integralrepresentation}
q(\zeta)=a+b  \zeta +\int_\R \left( \frac{1}{t-\zeta}-\frac{t}{1+t^2} \right)\mathrm{d}\nu(t)\quad \forall \zeta\in\mathbb \C\setminus \R. 
\end{equation}
\item[(C)]
The Nevanlinna kernel $N_f$ is positive definite. 
\end{itemize}
Moreover, if $q$ is a Herglotz-Nevanlinna function, then the absolutely continuous part of $\nu$ is given by 
\begin{align*}
       \mathrm{d}\nu_{abs}(t) = \frac{1}{\pi} \lim_{\epsilon \downarrow 0} \mathrm{Im}(q(t+\mathrm{i}\epsilon))  \mathrm{d} \lambda(t)= \frac{1}{2\mathrm{i}   \pi}  \left(\lim_{\epsilon \downarrow 0} q(t+\mathrm{i}\epsilon)-\lim_{\epsilon \downarrow 0} q(t-\mathrm{i}\epsilon)\right)
\end{align*}
Here, the first equality can be found in \cite[Theorem 3.27]{teschl2009mathematical} and the second follows from the fact that 
\begin{align*}
    \mathrm{Im}(q(t+\mathrm{i}\epsilon))=\frac{1}{2\mathrm{i}}\cdot(q(t+\mathrm{i}\epsilon)-q(t-\mathrm{i}\epsilon)).
\end{align*}
\end{proposition}
We denote the class of Herglotz-Nevanlinna functions by $\mathcal{N}_{0}$.
It is well known that 
Herglotz-Nevanlinna functions are of bounded type, i.e. that   $\mathcal{N}_{0} \subset \mathcal{N}_{sym}$. The integral representation in item (B) defines  a minimal realization in the following way (see e.g. \cite[Section 2.4]{CN05856002820060101}): 
\begin{proposition}
Let $q$ be a Herglotz-Nevanlinna function, $U_{ext}(q)$ its extended domain and $(a,b,\nu)$ as in Proposition \ref{Prop:Herglotz}. We define a Hilbert space $\C_b$ as
\begin{align*}
\begin{cases}
\C_b =\C, \quad [\cdot,\cdot]_{\C_b} \vcentcolon = b  [\cdot,\cdot]_\C \quad & \text{ if } b\neq 0 
\\
\C_b =\{0\} \quad & \text{ if } b= 0.
\end{cases}
\end{align*}
Then the following holds:
\begin{itemize}
\item
If $b=0$, then $q$ has a minimal realization $(\mathcal{K},A,\phi)$ given by 
\begin{align*}
    &\mathcal{K}=L^2(\R, \nu) \\  
    &A=\{(f,g) \in L^2(\R, \nu) \times L^2(\R, \nu) \ \big| \ g(t) \equiv t  f(t)  \} \\
    & \phi(\zeta)= \frac{1}{t-\zeta} \in L^2(\R, \nu).
\end{align*}
\item If $b>0$, then $q$ has a minimal realization given by 
\begin{align*}
    &\mathcal{K}=L^2(\R, \nu)\oplus \C_b, \\  
    &A=\left\{ \left(
    \begin{pmatrix}
        f \\
        c
    \end{pmatrix},
    \begin{pmatrix}
        g \\
        d
    \end{pmatrix}
    \right)
     \in L^2(\R, \nu)\oplus \C_b \times L^2(\R, \nu)\oplus \C_b \ \big| \ 
     g(t)\equiv t  f(t), \ c=0  \right\},  \\
    &\phi(\zeta)=\begin{pmatrix}
        \frac{1}{t-\zeta} \\
        1
    \end{pmatrix} \in L^2(\R, \nu)\oplus \C_b.
\end{align*}
\end{itemize}
The relation $A$ is an operator if and only if $b=0$. In either case it holds that $\varrho(A)=U_{ext}(q)$. 
\end{proposition}
Consequently, we can use Proposition \ref{Prop:RKSIso} to define a (unique) modelling reproducing kernel Hilbert space for Herglotz-Nevanlinna functions as it was done in e.g. \cite[Theorem 5]{DeBranges}:
\begin{proposition} \label{Prop:RealHnfunctions}
    Let $q$ be a Herglotz-Nevanlinna function. Then $q$ has a modeling reproducing kernel space $\mathcal{L}(q)$, which is described by the isomorphism
    \begin{align*}
        F \vcentcolon \ L^2(\R,\nu)\oplus \C_b \rightarrow \mathcal{L}(q), \quad x \mapsto F(x), 
        \quad F(x)(\zeta)=[x,\phi(\overline{\zeta})]_{L^2(\R,\nu)\oplus \C_b}.
    \end{align*}
    In particular, it holds that
    \begin{align*}
        F\left(\begin{pmatrix}
        f \\
        c
    \end{pmatrix}\right)(\zeta)=\left[\begin{pmatrix}
        f \\
        c
\end{pmatrix},
\begin{pmatrix}
        \frac{1}{t-\overline{\zeta}} \\
        1
\end{pmatrix}\right]_{L^2(\R,\nu)\oplus \C_b}
        = \int_{\R} \frac{f(t)}{t-\zeta}  \mathrm{d}\nu(t)+ b   c
    \end{align*}
    is the sum of a Cauchy transform and a constant, which means that $\mathcal{L}(q)\subset \mathcal{N}^+(\C \setminus \R)$, see e.g. \cite[Proposition 5.1.1.]{77570320060101}. 
\end{proposition}

\subsection{$\mathcal{B}(h)$-spaces} \label{Sec:Bh spaces}
As outlined in the introduction, we base our construction on a generalization of Helson's representation theorem for real complex functions. While it was sufficient to work with inner functions for that specific class,  we need to work with the class $\mathcal{S}_0$ in the general case:
\begin{definition}
Let $h \in \mathcal{N}(\C\setminus \R)$. Then $h$ is in the class $\mathcal{S}_0$ if   
\begin{align*}
h_{|\C^+} \in \mathcal{H}^\infty(\C^+) \setminus \{0\}, \quad 
\|h_{|\C^+}\|_{\mathcal{H}^\infty(\C^+)} \leq 1, \quad \text{ and } 
h_{|\C^-} =\frac{1}{SR(h_{|\C^+})}.
\end{align*}
Note that $h$ has a pole at $\zeta \in \C^-$ if and only if $h(\overline{\zeta})=0$.
\end{definition}
It is clear that any given function $f \in \mathcal{H}^\infty(\C^+)\setminus \{0\}$ satisfying $\|f\|_{\mathcal{H}^\infty(\C^+)} \leq 1$ generates a function $\tilde{f}$ in $\mathcal{S}_0$ by setting
\begin{align*}
\tilde{f}=f \ \text{ on $\C^+$  and } \
\tilde{f}_{|\C^-} =\frac{1}{SR(f_{|\C^+})}.
\end{align*}
In this case, we say that $\tilde{f}$ is generated by $f$. It should be noted  that if $f$ is an inner function, then its extension to $\C^-$ coincides with its pseudocontinuation. 
\par\smallskip
The connection between the class of Nevanlinna functions and the class $\mathcal{S}_0$ is given by the Cayley transform. Indeed, it is straightforward to check that 
\begin{align*}
    C \vcentcolon \mathcal{N}_0 \setminus \{ q \in  \mathcal{N}_0 \vcentcolon q_{|\C^+}\equiv \mathrm{i} \} \rightarrow \mathcal{S}_0 \setminus \{h\equiv -1\}, \quad q \mapsto C(q)= \frac{1+iq}{1-iq}
\end{align*}
is a bijection. 
\par\smallskip
The class  $\mathcal{S}_0$ was originally introduced by de Branges  \cite{edsjsr.10.2307.199375419630301} and has been studied on multiple occasions afterwards \cite{Alpay1},\cite{edssjs.3E774ADC19840901}. 
Naturally associated with  a function in $\mathcal{S}_0$ is the  reproducing kernel Hilbert space described in the next Theorem:  
\begin{theorem}[DeBranges]\label{Thrm:Basic1}
Let $h\in \mathcal{S}_0$. Then there exists a unique reproducing kernel Hilbert space  on $U_{ext}(h)$ and kernel 
\begin{align*}
    &s_h(z, w) \vcentcolon U_{ext}(h) \times U_{ext}(h) \rightarrow \C \\
    &s_h(z, w)=\frac{1-h(z)\overline{h(w)}}{- \mathrm{i}  (z - \overline{w})}, \  z \neq \overline w, \quad
s_h(w,\overline{w})=\frac{1}{-\mathrm{i} h(w)}  h'(w).
\end{align*}
We denote this space by $\mathcal{B}(h)$.
Moreover,  the difference-quotient  operator $D_w$ is well-defined and bounded on $\mathcal{B}(h)$ for every $w \in U_{ext}(h)$ such that $h(w)\neq 0$.
\end{theorem}
One proof can be found in \cite[Section 3]{edssjs.3E774ADC19840901}.  We will give another proof of  this theorem now.  In fact, we will prove slightly more for the sake of later referencing: 
\begin{theorem}\label{Thrm:BhLq}
    Let $h \in \mathcal{S}_0 \setminus \{h\equiv-1\}$ and define $q=C^{-1}(h)$. Then the reproducing kernel space given by 
    \begin{align*}
        \frac{1+h(\zeta)}{\sqrt{2}}  \mathcal{L}(q) \subset \mathcal{N}(\C \setminus \R)
    \end{align*}
    is the space $\mathcal{B}(h)$, as described in Theorem \ref{Thrm:Basic1}. Furthermore, for any $w  \in U_{ext}(h)$ such that $h(w)\neq 0$ it holds that
    \begin{align*}
    D_w(h)=
    \frac{ h(\zeta)- h(w)}{\zeta-w} \in \mathcal{B}(h).
    \end{align*}
\end{theorem}
\begin{proof}
We calculate
\begin{align*}
     \frac{1+h(\zeta)}{\sqrt{2}} & N_q(\zeta,w)   \frac{1+\overline{h(w)}}{\sqrt{2}}= \frac{1}{2  (\zeta-\overline{w})}  (1+h(\zeta))  (q(\zeta)-\overline{q(w)})  (1+\overline{h(w)})\\
     &= \frac{1}{2  (\zeta-\overline{w})}  (1+h(\zeta)) 
     \left(\mathrm{i}  \frac{1- h(\zeta)}{1+h(\zeta)}+\mathrm{i}  \frac{1- \overline{h(w)}}{1+\overline{h(w)}} \right)
       (1+\overline{h(w)})\\
     &=\frac{1}{-2  \mathrm{i}  (\zeta-\overline{w})}  
     ((1-h(\zeta)) (1+\overline{h(w)}) +(1+h(\zeta))  (1- \overline{h(w)}))=s_h(\zeta, w). 
\end{align*}
Consequently, we see that the Nevanlinna kernel $N_q$ and $s_h$ are related by the multiplication with the function $\frac{1+h(\zeta)}{\sqrt{2}}$. According to  Proposition \ref{Prop:Conj}, this means that 
\begin{align}
    \mathcal{B}(h)=\frac{1+h}{\sqrt{2}} \cdot \mathcal{L}(q) \subset \mathcal{N}(\C \setminus \R).
\end{align}
Next, let $w \in   U_{ext}(h)$ such that $h(w)\neq 0$. Since $h$ is  analytic at $\overline{w}$ and $h(w) \overline{h(\overline{w})}=1$ by construction,  we conclude that 
\begin{align}\label{For:WichtigesElement}
    -D_w(h)(\zeta)=-
    \frac{ h(\zeta)- h(w)}{\zeta-w}= -\mathrm{i}  h(w) \frac{ 1- h(\zeta)  \overline{h(\overline{w})}}{-\mathrm{i}  (\zeta-w)} =  h(w)  s_{h}(\zeta,\overline{w})\in \mathcal{B}(h).
\end{align}
Finally, consider $f \in \mathcal{B}(h)$. We can express $f$ as $f=(h+1) \cdot \tilde{f}$, where $\tilde{f}\in \mathcal{L}(q)$. Using the multiplication formula for $D_w$, we infer that
\begin{align*}
    D_w(f)=(h+1) \cdot D_w(\tilde{f})+ \tilde{f}(w) \cdot D_w(h+1) = (h+1) \cdot D_w(\tilde{f})+ \tilde{f}(w) \cdot D_w(h).
\end{align*}
Here, we have used the fact that the difference quotient doesn't ``see'' constants. 
Since $D_w$ is a bounded operator on $\mathcal{L}(q)$ and $D_w(h) \in \mathcal{B}(h)$, it follows that $D_w(f) \in \mathcal{B}(h)$. By the closed graph theorem $D_w$ is also bounded, which completes the proof. 
\end{proof}
A careful reading of the proof  reveals that the boundedness of $D_w$ is equivalent to $D_w(h) \in \mathcal{B}(h)$. We will discuss this relationship  at a later stage. Next, we clarify the relationship between $\mathcal{B}(h)$ spaces and the more widely  known de Branges-Rovnak spaces, which  are constructed as follows \cite[Section 18.4]{Fricain_Mashreghi_2016}:
\begin{proposition}
Let $f \in \mathcal{H}^\infty(\C^+)$ with $\|f\|_{\mathcal{H}^\infty(\C^+)} \leq 1$ and define the kernel function
\begin{align*}
    k_{f} \vcentcolon \C^+ \times \C^+ \rightarrow \C, \quad
    k_{f}(\zeta,w)= \frac{1-f(\zeta) \overline{f(w)}}{-\mathrm{i}  (\zeta-\overline w)}.
\end{align*}
Then this kernel generates a reproducing kernel Hilbert space $\mathcal{H}(f)$ on $\C^+$. Moreover,  it holds that $\mathcal{H}(f) \subset \mathcal{H}^2(\C^+)$ as sets.
\end{proposition}
The space $\mathcal{H}(f)$ is called de Branges-Rovnak space or model space if $f$ is an inner function. 
\par\smallskip
This raises the question, when the restriction to $\C^+$ is an isomorphism from $\mathcal{B}(h)$ to $\mathcal{H}(h_{|\C^+})$:
\begin{proposition}
\label{Prop:Rovnak}
Let $h\in \mathcal{S}_0$ be extreme, which means that  
\begin{align*}
\int_{\R} \frac{\log(1-|\lim_{\epsilon \downarrow 0} h(x+\mathrm{i}\epsilon)|)}{1+x^2} \mathrm{d} \lambda(x)=-\infty.
\end{align*}
Then the restriction operator 
\begin{align*}
    \mathrm{Res} \vcentcolon \mathcal{B}(h) \rightarrow \mathcal{H}(h_{|\C^+}), \quad f \mapsto f_{|\C^+}
\end{align*}
is an isometric isomorphism. 
\end{proposition}
\begin{proof}
Let $h\in \mathcal{S}_0$ be extreme  and consider the associated 
    Herglotz-Nevanlinna function $q = C^{-1}(h)$ given by the Cayley transform. Recall that in this case it holds that 
    \begin{align*}
    \frac{1+h}{\sqrt{2}}  \mathcal{L}(q)= \mathcal{B}(h).
    \end{align*}
    Then, since $h$ is extreme, it follows that  
    \begin{align*}
        \overline{\mathrm{span}} \{N_q(\cdot,w) \vcentcolon w \in \C^+ \}= \mathcal{L}(q),
    \end{align*}
    see \cite[Theorem 7.3]{S002212361200435120130215}. 
    This implies that the image 
    \begin{align*}
        \frac{1+h}{\sqrt{2}} \cdot  \mathrm{span} \{N_q(\cdot,w) \vcentcolon w \in \C^+ \}= 
        \mathrm{span}\{s_h(\cdot,w) \vcentcolon w \in \C^+ \}\subset \mathcal{B}(h)
    \end{align*}
    is also dense in $\mathcal{B}(h)$. Here, we have used that
    the multiplication by $\frac{1+h}{\sqrt{2}}$ defines an unitary operator and the following identity: 
    \begin{align*}
         \frac{1+h(\zeta)}{\sqrt{2}} & N_q(\zeta,w)   \frac{1+\overline{h(w)}}{\sqrt{2}}=s_h(\zeta, w). 
    \end{align*}
    Now consider the following operator: 
    \begin{align*}
        T \vcentcolon  \mathrm{span} \{s_h(\cdot,w) \vcentcolon w \in \C^+ \} \subset \mathcal{B}(h) \rightarrow \mathcal{H}(h_{|\C^+}), \quad f \mapsto f_{|\C^+}.
    \end{align*}
    Note that $T$ is well-defined because the restriction of the kernel function $s_h(\cdot,w)$ of $\mathcal{B}(h)$ is just the kernel function $k_{h_{|\C^+}}(\cdot,w)$ of $\mathcal{H}(h_{|\C^+})$. In particular, it follows that $T$ has dense range. Moreover, note that the reproducing property implies the following identity:
    \begin{align*}
        [s_h(\cdot,w),s_h(\cdot,x)]_{\mathcal{B}(h)}
        =\frac{1-h(x)\overline{h(w)}}{- \mathrm{i}  (x - \overline{w})}=
        [k_{h_{|\C^+}}(\cdot,x),k_{h_{|\C^+}}(\cdot,w)]_{\mathcal{H}(h_{|\C^+})} \quad \forall x,w \in \C^+.
    \end{align*}
    From there it is straightforward  to verify that $T$ is an isometry and that the continuous extension of $T$ is indeed the restriction operator. 
\end{proof}
We see that there is a subtle difference between the two sided space $\mathcal{B}(h)$ and the one sided de Branges-Rovnak space. This difference will turn out to be crucial, which is why take the lower half plane into account and consider functions in $\mathcal{N}_{sym}$ instead of $\mathcal{N}(\C^+)$. 
\par\smallskip
Next, we note that the class $\mathcal{S}_0$ has a multiplicative structure. Specifically, if $h_1$ and $h_2$ are elements of $\mathcal{S}_0$, then so is $h_1\cdot h_2$. This enables us to establish a divisibility theory analogous to the one for inner functions (Section \ref{Sec:BoundedType}):
\begin{definition}
Let $h_1$ and $h_2$ in $\mathcal{S}_0$. Then we say that 
\begin{itemize}
    \item the function $h_1$ divides $h_2$ if there exists a function $h_3\in \mathcal{S}_0$ such that $h_2=h_1 \cdot h_3$.
    \item the functions  $h_1$ and $h_2$ are relatively prime, if there exists no non-trivial  function in $\mathcal{S}_0$ which divides both $h_1$ and $h_2$. Here a function $h \in \mathcal{S}_0$ is non-trivial, if $h_{|\C^+}$ is not a constant function of modulus one.
\end{itemize}
\end{definition}
Similarly to the case of inner functions, we can characterize when two functions are relatively prime:
\begin{proposition}
    Let $h_1$ and $h_2$ be elements of $\mathcal{S}_0$ and let 
    $(h_i)_{|\C^+}=V_i \cdot O_i, \ i \in
    \{1,2\}$ be the inner outer decomposition. Then $h_1$ and $h_2$ are relatively prime if and only if 
    \begin{itemize}
        \item the inner functions $V_1$ and $V_2$ are relatively prime
        \item and the absolutely continuous measures $\log|O_1| \mathrm{d} \lambda$ and $\log |O_2| \mathrm{d} \lambda$ on $\R$ representing $ O_1$ and $O_2$ are mutually singular.
    \end{itemize}
\end{proposition}
\begin{proof}
    It is straightforward to verify that  $h_1$ and $h_2$ are relatively prime if $h_1$ and $h_2$ satisfy (i) and (ii). 
    \par\smallskip
    Conversely, let $h_1$ and $h_2$ be elements of $\mathcal{S}_0$, and assume that  $h_1$ and $h_2$ are relatively prime. If $V_1$ and $V_2$ are not relatively prime, then we can find an inner function $V_3$ that divides both $V_1$ and $V_2$. In this case, $h_3$ generated by $V_3$ will divide both $h_1$ and $h_2$, leading to a contradiction.
    \par\smallskip
    Now, consider the situation where $\log|O_1|\leq 0$ and $\log |O_2|\leq 0$ representing $ O_1$ and $O_2$ are not mutually singular. In this case,  we define the outer function 
    \begin{align*}
 O_3(\zeta) \vcentcolon = \exp \left( \frac{-\mathrm{i}}{\pi} \int_\R \left( \frac{1}{t-\zeta}-\frac{t}{1+t^2} \right) \min\{\log|O_1|, \log|O_2|\}  \mathrm{d}\lambda(t) \right),
    \end{align*}
    Since for $\lambda$ almost all $x \in \R$ it holds that  $\log|O_3(x)|=\min\{\log|O_1(x)|, \log|O_2(x)|\} \leq  0$ , we conclude that $\|O_3\|_{\mathcal{H}^\infty(\C^+)}\leq 1$. 
   The function $h_3\in \mathcal{S}_0$ generated by $O_3$ will then divide both $h_1$ and $h_2$, leading to a contradiction once again. This concludes the proof.
\end{proof}

Finally, we will need the following description of the boundary behaviour of functions in $\mathcal{B}(h)$.
\begin{proposition}\label{Prop:BoundaryBehaviour}
    Let $h\in \mathcal{S}_0$ and set 
    \begin{align*}
        P_h=\left\{  x \in \R \vcentcolon \lim_{\epsilon \downarrow 0} h(x+\mathrm{i}\epsilon)
    =\lim_{\epsilon \downarrow 0} h(x-\mathrm{i}\epsilon) \right\}. 
    \end{align*}
     Then for every 
     $g\in \mathcal{B}(h)$ it holds that 
\begin{align*}
    \lim_{\epsilon \downarrow 0} g(x+\mathrm{i}\epsilon)
    =\lim_{\epsilon \downarrow 0} g(x-\mathrm{i}\epsilon) \quad  \text{ for $\lambda$ almost all $x\in P_h$}. 
\end{align*}
\end{proposition}
\begin{proof}
Since the space $\mathcal{B}(-1)$ is trivial, we can assume without  loss of generality that $h \not\equiv -1$. This allows us to consider the Herglotz-Nevanlinna function $q$ defined as the inverse Cayley transform of $h$, i.e.  
\begin{align}\label{For:Unwichtig1}
    q=C^{-1}(h)=\mathrm{i}  \frac{1-h}{1+h}.
\end{align}
Let $\nu$ be the representing measure of $q$ given by Proposition \ref{Prop:Herglotz}. Then we can represent an arbitrary element $g \in \mathcal{B}(h)$ as
\begin{align*}
g (\zeta)= (h(\zeta)+1)  \left(\int_{\R} \frac{\tilde{g}(t)}{t-\zeta} \mathrm{d}\nu(t)+c \right)
\end{align*}
for some $\tilde{g} \in L^2(\R, \nu)$ and $c \in \C$. The boundary limits from above and below   clearly coincide for  $h+1$ and constants on $P_h$. Therefore, we only have to treat the Cauchy-integral  
\begin{align*}
    m(\zeta)=\int_{\R} \frac{\tilde{g}(t)}{t-\zeta} \mathrm{d}\nu(t).
\end{align*}
Now if $\tilde{g}$ is a non-negative function, then $m$ is a Herglotz-Nevanlinna function with representing measure  $\tilde{g} \mathrm{d}\nu$ (see \cite[Theorem 2.4.2]{LugerOu}). Then we conclude that (Proposition \ref{Prop:Herglotz})
\begin{align*}
       (\tilde{g}\mathrm{d}  \nu)_{abs}(t) = \frac{1}{2\mathrm{i}   \pi}  \left(\lim_{\epsilon \downarrow 0} m(t+\mathrm{i}\epsilon)-\lim_{\epsilon \downarrow 0} m
(t-\mathrm{i}\epsilon)\right).
\end{align*}
The same formula is also true for arbitrary functions $g\in L^2(\R,\mathrm{d}\nu)$. In that case, we decompose $g$ into its positive and negative real and imaginary parts and apply the same argument after pulling out $\mathrm{i}$ and minus signs if necessary.
Finally, we calculate the absolutely continuous part of $\mathrm{d}\nu$  using Proposition \ref{Prop:Herglotz}  
\begin{align*}
    (\mathrm{d}\nu)_{abs}(t)&=\frac{1}{2\mathrm{i}   \pi}  \left(\lim_{\epsilon \downarrow 0} q(t+\mathrm{i}\epsilon)-\lim_{\epsilon \downarrow 0} q
(t-\mathrm{i}\epsilon)\right)\\
&= 
\frac{1}{2\mathrm{i}   \pi}  \left(\lim_{\epsilon \downarrow 0} \mathrm{i}  \frac{1-h(t+\mathrm{i}\epsilon)}{1+h(t+\mathrm{i}\epsilon)}- \lim_{\epsilon \downarrow 0} \mathrm{i}  \frac{1-h(t-\mathrm{i}\epsilon)}{1+h(t-\mathrm{i}\epsilon)} \right).
\end{align*}
This means that $\mathrm{d}\nu_{abs}(t)=0$ on $P_h$ and consequently $(\tilde{g}\mathrm{d}  \nu)_{abs}(t)=0 \ \lambda \text{ a.e. on } \R$, which concludes the proof. 
 \end{proof}

\subsection{Boundary triples and symmetric operators}
In this last preliminary section, we discuss symmetric operators and their boundary mappings, which have been extensively studied in the past \cite{RN07186998219990101}, \cite{Behrndt2020}. These concepts will play no role in the actual construction of minimal realizations for functions of bounded type, but they will be crucial for characterizing the appearing relations afterward.  As always, we start with the definition: 
\begin{definition}
    Let $S$ be a closed operator in a Krein space $\mathcal{K}$. Then $S$ is called symmetric if $S \subset S^+$.
\end{definition}
Boundary mappings are then defined as follows:
\begin{definition}
Let $S$ be a symmetric operator  on a Krein space $\mathcal{K}$. A triple $(\mathcal{H},\Gamma_0,\Gamma_1)$ is called boundary triple for $S^+$, if $\mathcal H$ is a Hilbert space and $\Gamma_0$ and $\Gamma_1$ are bounded mappings from $S^+$ into $\mathcal H$ such that
\begin{itemize}
    \item the mapping $\Gamma_0 \times \Gamma_1$ into $\mathcal H \times \mathcal H$ is surjective 
    \item the abstract Lagrange identity holds: 
    \begin{align*}
        [g,h']_{\mathcal{K}}-[h,g']_{\mathcal{K}}= (\Gamma_1((h,g)),\Gamma_0((h',g')))_{\mathcal{H}}
        -(\Gamma_0((h,g)),\Gamma_1((h',g')))_{\mathcal{H}}
        \quad \forall (h,g), (h',g') \in S^+.
    \end{align*}
\end{itemize}
\end{definition}
Naturally associated to such a boundary triple is the self-adjoint linear relation $A \vcentcolon =\mathrm{ker}(\Gamma_0)$. In what follows, we will always assume that $\varrho(A)\neq \emptyset$ and define the defect subspaces 
\begin{align*}
\mathcal{F}_w\vcentcolon &= \{ g\in \mathcal{K}\ | \ (g,w  g) \in S^+ \}=\mathrm{ker}(S^+-w)\\
\tilde{\mathcal{F}}_w \vcentcolon &= \{ (g,w  g) \ | \ g \in \mathcal{F}_w \} \subset \mathcal{K} \times \mathcal{K} 
        .
\end{align*}
The relation 
$S^+$ can be decomposed as $S^+=A \oplus \tilde{\mathcal{F}}_z$ for $z \in \varrho(A)$, which can be used to show that
\begin{align*}
    (\Gamma_0)_{|\tilde{\mathcal{F}}_z} \vcentcolon \tilde{\mathcal{F}}_z \rightarrow \mathcal H
\end{align*}
is a bijective map. This allows us to define the following objects: 
\begin{definition}
    Let $S$ be a symmetric operator in a Krein space $\mathcal{K}$, $(\mathcal{H},\Gamma_0,\Gamma_1)$ a boundary triple and assume that $\varrho(A) \neq \emptyset$, where $A=\mathrm{ker}(\Gamma_0)$. Then we define 
    \begin{itemize}
        \item the Weyl solution $\gamma \vcentcolon \varrho(A) \rightarrow \mathscr{L}(\mathcal{H},\mathcal{K})$ as 
                \begin{align*}
                \gamma(z) \vcentcolon = P_1 \left((\Gamma_0)_{|\tilde{\mathcal{F}}_z}\right)^{-1} \vcentcolon \mathcal H \rightarrow \mathcal K,
                \end{align*}
                where $P_1$ denotes the projection onto the first component of $\mathcal K \times \mathcal K$. This function satisfies 
            \begin{align*}
                \gamma(w)=(I+(w-z)(A-w)^{-1}) \gamma(z) \quad \forall w,z \in \varrho(A).
            \end{align*}
        \item the Weyl function $m \vcentcolon \varrho(A) \rightarrow \mathscr{L}(\mathcal{H}) $ as 
        \begin{align*}
            m(z) \vcentcolon = \Gamma_1\left((\Gamma_0)_{|\tilde{\mathcal{F}}_z}\right)^{-1}.
        \end{align*}
    \end{itemize}
\end{definition}
If $\mathcal{K}$ is a Hilbert space, then the dimension of $\mathcal{F}_w$ is constant on $\C^+$ and $\C^-$ which allows us to define the deficiency index as $(n,m)=(\mathrm{dim}(\mathcal{F}_\mathrm{i}),\mathrm{dim}(\mathcal{F}_{-\mathrm{i}}))$.
Boundary triples on Hilbert spaces are closely related to realizations of Herglotz-Nevanlinna functions \cite[Section 2]{article55}:
\begin{proposition}\label{Prop:Symop1}
Let $q$ be a Herglotz-Nevanlinna function and $\mathcal{L}(q)$ its modeling reproducing kernel space. Then the multiplication operator by the independent variable, which we denote by $S_q$, is a closed, simple  symmetric operator with deficiency index $(1,1)$. Its adjoint is given by 
\begin{align*}
    S^+=\{ (f,g) \in \mathcal{L}(q) \times  \mathcal{L}(q) \ \big| \exists c,d \in \C \vcentcolon  g(\zeta)-\zeta  f(\zeta)=c-d  q(\zeta) \}.
\end{align*}
Finally, a boundary triple $(\C,\Gamma_0,\Gamma_1)$ is defined by $\Gamma_0=c$ and $\Gamma_1=d$ and the associated Weyl function is $q$.  
\end{proposition}
Recall that a symmetric operator in a Hilbert space $\mathcal H$ is simple or completely non-self-adjoint, if there exists no non-trivial orthogonal decomposition $\mathcal H =\mathcal{H}_1\oplus \mathcal{H}_2$ such that $S$ decomposes as $S=S_1 \oplus S_2$ and $S_2$ is self-adjoint in $\mathcal{H}_2$.
This Proposition has the following inverse \cite[Section 2]{article55}: 
\begin{proposition}\label{Prop:Symop2}
    Let $S$ be a closed, simple symmetric operator in a Hilbert space $\mathcal H$ with deficiency index $(1,1)$. Then there exists a boundary triple $(\C,\Gamma_0,\Gamma_1)$. Moreover, given such a boundary triple, then the Weyl function $q$ is a Herglotz-Nevanlinna function and the operator 
    \begin{align*}
        F \vcentcolon \mathcal H \rightarrow \mathcal L(q), \quad x \mapsto [x, \gamma(\overline{\zeta})]_{\mathcal H} \quad \zeta \in \C \setminus \R
    \end{align*}
    is an isometric isomorphism such that $S=F^{-1} S_q F$. 
\end{proposition}
We point out, that for fixed $x\in \mathcal{H}$ the function $\zeta \mapsto  [x, \gamma(\overline{\zeta})]$ is an element of $\mathcal L(q)$ and thus in particular of $\mathcal{N}(\C\setminus \R)$. This will be important later on. 
\color{black}

\section{Minimal realizations for functions of bounded type}
\label{Sec:Construction}
Now we turn to the proof of our main theorem: 
\begin{theorem}\label{Thrm:Main1}
    Let $f \in \mathcal{N}_{sym}$. Then there exists a modeling reproducing kernel space $\mathcal{L}(f)$ on $U_{ext}(f)$ for $f$. In other words, $f$ has a minimal  realization $(\mathcal{L}(f),A,v)$ with $\varrho(A)=U_{ext}(f)$. 
\end{theorem}
As outlined in the introduction, we base our construction on a suitable pair of $\mathcal{B}(h)$ spaces. In order to do so, we  establish a connection between a function $f \in \mathcal{N}_{sym}$ and two functions  $h_1, h_2 \in \mathcal{S}_0$ via the Cayley transform:
\begin{proposition}\label{Prop:MainTech}
    Let $f \in \mathcal{N}_{sym}$ and assume that $f$ is non-constant on $\C^+$. Then there exist functions $h_1, h_2 \in \mathcal{S}_0$ which are relatively prime  such that 
    \begin{align*}
        f = \mathrm{i}  \frac{h_2-h_1}{h_2+h_1}.
    \end{align*}
\end{proposition}

\begin{proof}
    Let $f\in \mathcal{N}_{sym}$ be such that $f_{|\C^+}$ is non-constant, and consider its Cayley transform  $C(f) \in \mathcal{N}(\C \setminus \R)$.
    We can express $C(f)_{|\C_+}\in \mathcal{N}(\C^+)$ as 
\begin{align*}
    C(f)_{|\C_+}= \frac{V_1}{V_2} \cdot F = \frac{V_1}{V_2} \cdot \exp\left( \frac{-\mathrm{i}}{\pi} \int_{\R}  \left(\frac{1}{t-\zeta}-\frac{t}{1+t^2}\right)   m(t)  \mathrm{d} \lambda(t)  \right)
\end{align*}
where $V_1$ and $V_2$ are inner functions which are relatively prime and $F$ is the outer factor. 
The absolutely continuous measure $m(x)  \mathrm{d} \lambda$ can be decomposed into its positive and negative part, i.e. into $m_+(x)  \mathrm{d} \lambda $ and $m_-(x)  \mathrm{d} \lambda $, which are mutually singular. We  define outer functions $O_1$ and $O_2$ as
\begin{align*}
    O_1(\zeta)&= \exp\left( \frac{-\mathrm{i}}{\pi} \int_{\R}  \left( \frac{1}{t-\zeta}-\frac{t}{1+t^2} \right)  m_-(t)  \mathrm{d} \lambda(t)  \right) \\  
    O_2(\zeta)&= \exp\left( \frac{-\mathrm{i}}{\pi} \int_{\R}  \left( \frac{1}{t-\zeta}-\frac{t}{1+t^2} \right)    (-m_+(t))  \mathrm{d} \lambda(t) \right).
\end{align*}
Note that $\|V_i \cdot O_i\|_{\mathcal{H}^\infty(\C^+)} \leq 1$, since $\log|O_1(x)|=m_-(x)\leq 0$ and $\log|O_2(x)|=-m_+(x)\leq 0$.
Moreover, by construction it holds that
\begin{align*}
    C(f)_{|\C_+}= \frac{V_1    \cdot O_1}{V_2 \cdot O_2}.
\end{align*}
We define $h_1$ and $h_2$  as the elements in $\mathcal{S}_0$ generated by $V_i \cdot O_i$.  The functions $h_1$ and $h_2$   are relatively prime by construction. And since $C(f)_{|\C^-}=\frac{1}{SR(C(f)_{|\C^+})}$ (Proposition \ref{Prop:Cayley}), it follows that 
\begin{align*}
    C(f)= \frac{h_1}{h_2} \quad \text{ or equivalently } \quad
    f=C^{-1} (C(f)) = C^{-1}\left( \frac{h_1}{h_2} \right) =  \mathrm{i}  \frac{h_2-h_1}{h_2+h_1},
\end{align*}
which concludes the proof.
\end{proof}
In summary, every function $f\in \mathcal{N}_{sym}$ can be represented in terms of two functions in $\mathcal{S}_0$ which are relatively prime. This allows us to generalize Helson's representation theorem for  functions with real boundary values \cite[Theorem 3.1.]{articleRealpositive}:
\begin{theorem}\label{basic1.5}
Let $h_1$ and $h_2$ be elements of  $\mathcal{S}_0$. Then $h_1$ and $h_2$ are relatively prime if and only if 
\begin{align*}
    \mathcal{B}(h_1)\cap\mathcal{B}(h_2)=\{0\}.
\end{align*}
In particular, every function $f\in \mathcal{N}_{sym}$ which is non-constant on $\C^+$ has a representation of the form 
\begin{align*}
       f = \mathrm{i}  \frac{h_2-h_1}{h_2+h_1} \quad \text{ and } \quad
    \mathcal{B}(h_1)  \cap \mathcal{B}(h_2)= \{ 0\}.
\end{align*}
\end{theorem}
We will give a proof in Section 4. For the moment, we just mention that de Branges studied this  intersection problem in the 1960s and obtained some partial results in the ``only if'' direction \cite[Theorem 5 and 6]{edsjsr.10.2307.199375419630301}.  In Section 4, we will also introduce the following technical lemma crucial for addressing complexities associated with the real line:
\begin{lemma}\label{BoundaryBehaviour}
Let $h_1$ and $h_2$ be functions in $\mathcal{S}_0$ which are relatively prime. Then the function
\begin{align*}
     \frac{h_1}{h_2}
\end{align*}
can be analytically continued through an open interval $(a,b)\subset \R$ if and only if both $h_1$ and $h_2$ can be analytically continued through $(a,b)$.
\end{lemma}
For the moment, let's assume the preceding Theorem and Lemma to be true and proceed to construct a modeling reproducing kernel space based on these results. We are now almost ready to do so. However, there is a minor issue with the representation of a function $f \in \mathcal{N}_{sym}$ as 
\begin{align*}
    f = \mathrm{i}  \frac{h_2-h_1}{h_2+h_1}.
\end{align*}
Indeed, the function $f$ remains analytic even if $h_1$ or $h_2$ has a pole, in which case $f$ takes the values $\mathrm{i}$ or $-\mathrm{i}$. The poles (and zeros as well) of $h_1$ and $h_2$ will require additional consideration. To address these exceptional points, we will use the following Lemma:
\begin{lemma}\label{Lem:IsoPoints}
Let $f \in \mathcal{N}_{sym}$ and $(\mathcal{K},A,v)$  be a minimal realization of $f$. Moreover, fix  $\zeta_0 \in \C $. If  there exists an $\epsilon >0$ such that $B_\epsilon (\zeta_0) \subset  U_{ext}(f)$  and 
$B_\epsilon (\zeta_0) \setminus \{\zeta_0 \} \subset \varrho(A)$,  then $\zeta_0 \in \varrho(A)$. 
\end{lemma}
A proof can be found in  \cite[Section 6]{me2}. For the sake of completeness, we mention that the proof in \cite[Section 6]{me2} was formulated with the additional assumption that  $f$ has real boundary values. However, this assumption on $f$  was not used in the proof.
\begin{proof}[Proof of Theorem \ref{Thrm:Main1}]
Let $f\in \mathcal{N}_{sym}$. If $f$ is constant on $\C^+$, then either $f$ or $-f$ is a Herglotz-Nevanlinna function and Theorem \ref{Thrm:Main1} is clearly true. Consequently, we can assume without loss of generality that this is not the case, which allows us to find  $h_1$ and $h_2$ as described in Proposition \ref{Prop:MainTech}.
Consider the reproducing kernel Hilbert spaces 
\begin{align*}
    \mathcal{K}_+ \vcentcolon =\frac{\sqrt{2}}{h_1+h_2} \cdot \mathcal{B}(h_1) \subset \mathcal{N}(\C\setminus \R), \quad 
    \mathcal{K}_- \vcentcolon = \frac{\sqrt{2}}{h_1+h_2} \cdot \mathcal{B}(h_2) \subset \mathcal{N}(\C\setminus \R),
\end{align*}
We want to take the sum of these spaces, for which we first specify a suitable common domain. It is clear that both $\mathcal{K}_+$ and $\mathcal{K}_-$ consists of functions which are at least analytic on the set
\begin{align*}
    \Omega(f)\vcentcolon 
    = \{ \zeta \in \C \vcentcolon  h_1 \text{ and } h_2 \text{ are analytic at $\zeta$, } 
    h_1(\zeta)+h_2(\zeta)\neq 0 \text{ and }
    h_1(\zeta)  h_2(\zeta) \neq 0 \}.
\end{align*}
The zeros of $h_1$ and $h_2$ were excluded as well in order to ensure that $D_w$ acts boundedly on $\mathcal{B}(h_i)$ for every $w \in \Omega(f)$ (see Theorem \ref{Thrm:Basic1}). Note that $f$ is given by  
\begin{align*}
    f = \mathrm{i}  \frac{h_2-h_1}{h_2+h_1},
\end{align*}
which means that $\Omega(f)\subset U_{ext}(f)$.
Henceforth, we consider $\mathcal{K}_+$ and $\mathcal{K}_-$ as reproducing kernel spaces consisting of functions defined on $\Omega(f)$ by restricting the domains accordingly. We have seen in Proposition \ref{Prop:MainTech}, that  the intersection of $\mathcal{K}_+$ and $\mathcal{K}_-$  is trivial. Therefore, their   direct sum $\mathcal{L}(f)\vcentcolon = \mathcal{K}+ \oplus\mathcal{K}-$ together with the inner product 
\begin{align*}
   [f_++f_-,g_++g_-]_{\mathcal{L}(f)}\vcentcolon =
   [f_+,g_+]_{\mathcal{K}_+}-[f_-,g_-]_{\mathcal{K}_-} \quad \forall f_+,g_+ \in \mathcal{K}_+, \quad  f_-,g_- \in \mathcal{K}_-.
\end{align*}
is a reproducing kernel Krein space. Its kernel is $N_f(\zeta,w)$, because
\begin{align*}
\frac{\sqrt{2}}{h_1(\zeta)+h_2(\zeta)}& (s_{h_1}(\zeta,w)-s_{h_2}(\zeta,w) )  \frac{\sqrt{2}}{\overline{h_1(w)}+\overline{h_2(w)}} \\
&=\frac{\sqrt{2}}{h_1(\zeta)+h_2(\zeta)} \left(\frac{1-h_1(\zeta) \overline{h_1(w)}}{-\mathrm{i}  (\zeta-\overline w)}-\frac{1-h_2(\zeta) \overline{h_2(w)}}{-\mathrm{i} (\zeta-\overline w)} \right)  \frac{\sqrt{2}}{\overline{h_1(w)}+\overline{h_2(w)}}\\
&=\frac{\mathrm{i}}{\zeta-\overline w}  \left( \frac{2  (h_2(\zeta) \overline{h_2(w)}-h_1(\zeta) \overline{h_1(w)} ) }{(h_1(\zeta)+h_2(\zeta)) (\overline{h_1(w)}+\overline{h_2(w)})} \right) \\
&
=\frac{\mathrm{i}}{\zeta-\overline w}  \left(
\frac{h_2(\zeta)-h_1(\zeta)}{h_1(\zeta)+h_2(\zeta)}
+
\frac{\overline{h_2(w)}-\overline{h_1(w)}}{\overline{h_1(w)}+\overline{h_2(w)}} \right)
=\frac{f(\zeta)- \overline{f(w)}}{\zeta-\overline w}
=N_f(\zeta,w).           
\end{align*}
Next, we show that the difference quotient operator is well-defined and  bounded on this space. To this end, fix  $w \in \Omega(f)$ and pick some arbitrary element $g\in \mathcal{B}(h_1)$. This corresponds to an element 
\begin{align*}
    \frac{1}{h_1+h_2} \cdot g \in \mathcal{K}_+.
\end{align*}
We calculate using the multiplication formula for the difference quotient
\begin{align*}
    D_w\left(\frac{1}{h_1+h_2} \cdot g\right)
    =\frac{1}{h_1+h_2} \cdot  D_w(g)+g(w) \cdot D_w\left(\frac{1}{h_1+h_2}\right).
\end{align*}
Since $D_w$ is a bounded operator on $\mathcal{B}(h_1)$ we conclude that $\frac{1}{h_1+h_2} \cdot D_w(g) \in \mathcal{L}(f)$. Moreover, it  is easy to see that 
\begin{align}\label{For:Decomposition}
    D_w\left(\frac{1}{h_1+h_2}\right)=-\frac{1}{h_1+h_2} \cdot (D_w(h_1)+D_w(h_2))\cdot \frac{1}{h_1(w)+h_2(w)}
\end{align}
and we conclude that $D_w\left(\frac{1}{h_1+h_2}\right) \in \mathcal{L}(f)$ by   Theorem \ref{Thrm:BhLq}. 
The same argument shows that also $D_w(g) \in \mathcal{L}(f)$ for $g \in \mathcal{K}_-$. Consequently, the operator $D_w$ is everywhere defined on $\mathcal{L}(f)$ and is closed since point evaluations are continuous. This means that $D_w$ is even continuous by the closed graph theorem. 
\par\smallskip
In summary, we have constructed a modeling reproducing kernel space on $\Omega(f)$,
or equivalently a  minimal realization $(\mathcal{L}(f),A,N_f(\cdot, \overline{\zeta_0}))$ with $\Omega(f) \subset \varrho(A)$ and $\zeta_0 \in \Omega(f)$.  
Now let $\zeta \in U_{ext}(f)$. Then the following holds: 
\begin{itemize}
    \item If $\zeta \in \C \setminus \R$ such that $\zeta \notin \Omega(f)$, then either $h_1$ or $h_2$ has a zero or pole at $\zeta$.
    \item If $\zeta \in \R$, then $\zeta \in \Omega(f)$. Indeed, consider an interval $\zeta \in (a,b)$ on which $f$ is analytic. Then $f$ is real valued there. Consequently, it follows that $C(f)$ given by
    \begin{align*}
         C(f)=\frac{1+\mathrm{i}  f}{1-\mathrm{i} f} =\frac{h_1}{h_2}
    \end{align*}
    is analytic on $(a,b)$ and does not take the values $-1$ nor $0$ there. This implies that $h_1$ and $h_2$ are analytic on $(a,b)$ (Theorem \ref{basic1.5})  and that  $h_i \neq 0$ and $h_1+h_2 \neq 0$ on $(a,b)$.
\end{itemize}
Thus, it follows that $\Omega(f)$ falls short of $U_{ext}(f)$ only by a set of isolated points, namely the poles and zeros of $h_1$ and $h_2$. We apply Lemma $\ref{Lem:IsoPoints}$ to each of these points to conclude that $U_{ext}(f)=\varrho(A)$. This means that every function in $\mathcal{L}(f)$ can be analytically extended to $U_{ext}(f)$ and that $D_w$ is well defined and bounded on $U_{ext}(f)$. This concludes the proof. 
\end{proof}
The existence of the analytic continuation onto the exceptional points might seem mysterious  at first glance since it followed from an abstract argument. However, this can also be proven directly in the following way:
\begin{remark}
    Let $q_1$ be the Cayley transform of $h_1$, which is a Herglotz-Nevanlinna function. Then it follows that
    \begin{align*}
        \mathcal{K}_+= \frac{1}{h_1+h_2} \cdot \mathcal{B}(h_1) = \frac{h_1+1}{h_1+h_2} \cdot \mathcal{L}(q_1).
    \end{align*}
    From this equality it  can be read off that  any function $f \in \mathcal{K}_+$ has an analytic continuation at $\zeta$ even if $\zeta$ is a pole of $h_1$. Indeed the  function 
    \begin{align*}
        \frac{h_1+1}{h_1+h_2}
    \end{align*}
    is analytically continuable at such a pole  and $\mathcal{L}(q_1)$ consists of functions which are analytic on  $\C \setminus \R$. 
\end{remark}
\subsection{The key argument revisited }\label{GenBhspaces}
The argument made in the proof of Theorem \ref{Thrm:Main1} is a special case of the following general principle: 
\begin{theorem}\label{Thrm:GenPrin}
Let $\mathcal{K}$ be a reproducing kernel Krein space on a domain $\Omega$, and assume that $D_w$ is well defined and bounded for every $w\in \Omega$. Furthermore, let $f\in\mathcal{O}(\Omega)$ such that
\begin{itemize}
\item[(i)] the function $D_w(f)$ is an element of  $\mathcal{K}$ 
\item[(ii)] the function  $\frac{1}{f}$ is  meromorphic on $\Omega$.
\end{itemize}
In this case,  the difference-quotient operator $D_w$ is well-defined and bounded on the reproducing kernel Krein space $\frac{1}{f} \cdot \mathcal{K}$ for every $w\in \Omega \setminus \{ w \in \Omega \vcentcolon f(w)=0\}$. Note that we do not assume that $f \in \mathcal{K}$. 
\end{theorem}
\begin{proof}
    Let $\frac{1}{f} \cdot g \in \frac{1}{f} \cdot \mathcal{K}$ and calculate
    \begin{align*}
        D_w\left( \frac{1}{f} \cdot g \right)= \frac{1}{f} \cdot D_w(g)+g(w) \cdot  D_w\left( \frac{1}{f}\right).
    \end{align*}
    Since point evaluations are continuous and $\frac{1}{f} \cdot D_w(g)\in \frac{1}{f} \cdot \mathcal{K}$ by assumption, we conclude that $D_w$ is well defined and bounded on $\frac{1}{f} \cdot \mathcal{K}$ if and only if 
    \begin{align*}
        D_w\left( \frac{1}{f}\right)\in \frac{1}{f} \cdot \mathcal{K}, \text{ which is equivalent to } f \cdot D_w\left( \frac{1}{f}\right) = \frac{1}{f(w)} D_w(f)\in \mathcal{K}. 
    \end{align*}
     The latter statement is true by assumption. 
\end{proof}
This principle allows us to define realizations of Möbius transformations. More precisely, let $q$ be of bounded type and $\mathcal{L}(f)$ be a modeling reproducing kernel space. Then $\frac{1}{f} \cdot \mathcal{L}(f)$ is a  reproducing kernel Krein space with kernel
\begin{align*}
   \frac{1}{f(\zeta)} \cdot N_f(\zeta,w) \cdot \frac{1}{\overline{f(w)}}= N_{\frac{1}{-f}}(\zeta,w).
\end{align*}
Moreover, $D_w$ is bounded on $\frac{1}{f} \cdot \mathcal{L}(f)$ since $D_w(f)=N_f(\cdot,\overline{w}) \in \mathcal{L}(f)$, which means that this is a modeling reproducing kernel space for $\frac{1}{-f}$  This was known before \cite{me}, however this reasoning significantly simplifies the proof given there.
\par\smallskip
In the following, we describe yet another application of this principle. So far we were interested in reproducing kernel spaces associated to the Nevanlinna kernel generated by functions in $  \mathcal{N}_{sym}$. %We emphasize that the symmetry condition $f(\overline{\zeta})=\overline{f(\zeta)}$ is intrinsically linked to the kernel $N_{f}$. 
In the context of inverse scattering theory  \cite{Alpay1},  one often considers reproducing kernel  spaces corresponding to the kernel 
\begin{align*}
    s_v(\zeta, w)=\frac{1-v(\zeta)\overline{v(w)}}{- \mathrm{i}  (\zeta - \overline{w})}.
\end{align*}
In this  case, the natural symmetry condition for the function $v \in \mathcal{N}(\C \setminus \R)$ is 
\begin{align*}
        v_{|\C^-}=\frac{1}{SR(v_{|\C^+})}.
\end{align*}
Of course, if $v_{|\C^+}$ is analytic and bounded by one, then $v \in \mathcal{S}_0$ and we end up with a $\mathcal{B}(v)$ space. However, the methods developed in this article allow us to define such spaces for arbitrary symbols of bounded type,  and furthermore  show that they are invariant under the difference-quotient operator:
\begin{theorem}
    Let $v\in \mathcal{N}(\C \setminus \R)$ and  assume that $v$ satisfies the symmetry condition
    \begin{align*}
        v_{|\C^-}=\frac{1}{SR(v_{|\C^+})}.
    \end{align*}
    Then there are functions $h_1$, $h_2$ in $\mathcal{S}_0$ such that $v=\frac{h_1}{h_2}$ and $\mathcal{B}(h_1)\cap \mathcal{B}(h_2)=\{0\}$. Moreover, the space 
    \begin{align*}
        \mathcal{B}(v)\vcentcolon= \frac{1}{h_2} \cdot 
    ( \mathcal{B}(h_1) \oplus  -\mathcal{B}(h_2))
    \end{align*}
    is a reproducing kernel Krein space with  
    \begin{align*}
    \text{domain} \quad 
    &U_0(v) \vcentcolon = \{\zeta \in \C \vcentcolon  h_1 \text{ and } h_2 \text{ are analytic at } \zeta \text{ and } h_1(\zeta)h_2(\zeta)\neq 0 \} \\
    \text{and kernel} \quad 
     &s_v(\zeta, w) \vcentcolon U_0(v) \times U_0(v)  \rightarrow \C, \quad   s_v(\zeta, w)=\frac{1-v(\zeta)\overline{v(w)}}{- \mathrm{i}  (\zeta - \overline{w})}.
    \end{align*}
    Finally, $D_w$ is well defined and bounded for every $w\in U_0(v)$.
\end{theorem}
\begin{proof}
It is clear that we can find $h_1$ and $h_2$ as described above. Then the space $\mathcal{B}(h_1) \oplus  -\mathcal{B}(h_2)$ is a reproducing kernel Krein space consisting of functions which are analytic at least on $U_0$. The kernel of $\mathcal{B}(v)$ is 
$s_v(\cdot, w)$, because
\begin{align*}
    s_v(\cdot, w)=\frac{1}{h_2(\cdot)}
     ( s_{h_1}(\cdot,w) -s_{h_2}(\cdot,w))  \frac{1}{\overline{h_2(w)}}.
\end{align*}
Finally, the $D_w$ is well defined and bounded on this space, because $D_w$ is well defined and bounded on $\mathcal{B}(h_1) \oplus  -\mathcal{B}(h_2)$  and  $D_w(h_2) \in \mathcal{B}(h_1) \oplus  \mathcal{B}(h_2)$, see equation \eqref{For:WichtigesElement}. Note that we have used  Theorem \ref{Thrm:GenPrin} here. 
\end{proof}
If $v_{|\C^+}$ has unimodular boundary values, then 
the space $\mathcal{B}(v)$ was considered  in \cite[Section 6]{Alpay1} for the first time. However,  they did not  show that the difference quotient operator is bounded.

\section{The proof of Theorem \ref{basic1.5} and Lemma \ref{BoundaryBehaviour}} \label{Sec:Intersection}
\subsection{Intersections of $\mathcal{B}(h)$-spaces}
In this section, we give the missing proofs for Theorem \ref{basic1.5} and Lemma \ref{BoundaryBehaviour}. We begin with the intersection result, which is restated for convenience here:
\begin{proposition}
Let $h_1$ and $h_2$ be elements of $\mathcal{S}_0$. Then $h_1$ and $h_2$ are relatively prime if and only if
\begin{align*}
    \mathcal{B}(h_1)\cap\mathcal{B}(h_2)=\{0\}.
\end{align*}
\end{proposition}
\par\smallskip
First, we treat the case in which both $h_1$ and $h_2$ are generated by an inner function. If that is the case, then we
have seen in Proposition \ref{Prop:Rovnak} that $\mathcal{B}(h)$ is isomorphic to the model space associated to $h_{|\C^+}$. More precisely, the following holds:
\begin{proposition}\label{Schur:Basic2}
Let $h,h_2\in \mathcal{S}_0$ be generated by  inner functions $V$ and $V_2$ respectively.
Then the following holds:
\begin{itemize}
    \item[(i)] The restriction operator $\mathrm{R} \vcentcolon  \mathcal{B}(h) \rightarrow \mathcal{H}(V) \subset \mathcal{H}^2(\C^+)$ is an isometric isomorphism. Moreover, for every
    $g \in \mathcal{B}(h)$  the functions $g_{|\C^+}$ and  $g_{|\C^-} $ are pseudocontinuations of each other.
    \item[(ii)] Assume that $V$ and $V_2$ are relatively prime. Then it holds that $ \mathcal{H}(V) \cap  \mathcal{H}(V_2)=\{0\} $ and consequently also 
    \begin{align*}
        \mathcal{B}(h
)\cap \mathcal{B}(h_2)=\{0\}.
    \end{align*}
    \item[(iii)] Let $g \in \mathcal{B}(h) \setminus \{0\}$. Then there exists a function  $F \in  \mathcal{H}^2(\C^-)$ such that
    \begin{align*}
    g_{|\C^-} = \frac{F}{SR((V)_{|\C^+})}, \quad F \in  \mathcal{H}^2(\C^-).
    \end{align*}
    Moreover,  $g_{|\C^-}$ is not an element of $\mathcal{N}^+(\C^-)$.
\end{itemize}
\end{proposition}
\begin{proof}
Statement (i) is just a reformulation of  \cite[Proposition 7.14]{garcia2016introduction} and statement (ii) can be found in \cite[Corollary 5.9]{garcia2016introduction}.
Finally, statement (iii) can be found in \cite[Page 43]{garcia2013model}. Note that $g_{|\C^-}$ cannot be an element of $\mathcal{N}^+(\C^-)$  because of  
Proposition \ref{Prop:RegRes}.
\end{proof}
Now we will turn to the general case, in which the outer factors no longer have to be trivial. This is exactly the point at which it is important to take the lower half plane $\C^-$ into account. Indeed, much of what will be said  below is only true for $\mathcal{B}(h)$ spaces and  not for the respective 
de Branges-Rovnak spaces. 
We first prove an intersection result for purely outer functions: 
\begin{lemma}\label{Lem:IntOuter}
Let $h_1, h_2\in \mathcal{S}_0$ be generated by outer functions $O_1$ and $O_2$. If  $h_1$ and $h_2$ are relatively prime, then 
\begin{align*}
\mathcal{B}(h_1)\cap \mathcal{B}(h_2)=\{0\}.
\end{align*}
\end{lemma}
\begin{proof}
 If either $h_1$ or $h_2$ is trivial, then either $\mathcal{B}(h_1)$ or $\mathcal{B}(h_2)$ is trivial, and in this case  there is nothing to show.  Therefore, we can assume without loss of generality that both $h_1$ and $h_2$ are not trivial, which means that both $O_1$ and $O_2$ are non-constant.
 \par\smallskip
 Let $g \in \mathcal{B}(h_1)\cap \mathcal{B}(h_2)$ and recall that the measures 
$\log(|O_1|))  \mathrm{d} \lambda$ and $\log(|O_2|)  \mathrm{d} \lambda$ on $\R$ are mutually singular. We prove that $g$ has to be zero in three steps:

 \par\medskip
    \emph{1. The functions $g_{|\C^+}$ and $g_{|\C^-}$ are pseudocontinuations of each other.}  
 \newline
First, we observe that for $\lambda$ almost all $x \in \R$ 
\begin{align} \label{For:45809}
    \text{ either }  \lim_{\epsilon \downarrow 0} O_1(x+\mathrm{i}\epsilon) \in \T \text{ or }  \lim_{\epsilon \downarrow 0} O_2(x+\mathrm{i}\epsilon) \in \T. 
\end{align}
Indeed, the boundary limits of $O_1$ and $O_2$ exist  $\lambda$-almost everywhere and are smaller or equal $1$ $\lambda$ a.e. by construction. Moreover, if there existed a set $E\subset \R$ with positive measure on which both $|O_1|$ and $|O_2|$ are strictly smaller than $1$ simultaneously, then the densities $\log(|O_1|)$ and $\log(|O_2|)$ would not be mutually singular. 
\par\smallskip
Now fix a point $x \in \R$ such that $\lim_{\epsilon \downarrow 0} O_j(x+\mathrm{i}\epsilon)\in \T$ for some $j \in \{1,2\}$. Then $x$ belongs to set $P_{h_j}$, which was introduced in Proposition \ref{Prop:BoundaryBehaviour} . To demonstrate this, recall that $P_{h_j}$ was defined as 
\begin{align*}
    P_{h_j}=\left\{  x \in \R \vcentcolon \lim_{\epsilon \downarrow 0} h_j(x+\mathrm{i}\epsilon)
    =\lim_{\epsilon \downarrow 0} h_j(x-\mathrm{i}\epsilon) \right\}.
\end{align*}
Moreover, note that  $O_j$ coincides with $(h_j)_{|\C^+}$. Then we calculate 
    \begin{align*}
    \lim_{\epsilon \downarrow 0} h_j(x-\mathrm{i}\epsilon)= \frac{1}{\lim_{\epsilon \downarrow 0} \overline{h_j(x+\mathrm{i}\epsilon)}} = \lim_{\epsilon \downarrow 0} h_j(x+\mathrm{i}\epsilon),
\end{align*} 
where we have used the relation $(h_j)_{|\C^-} =\frac{1}{SR((h_j)_{|\C^+})}$ in the first equality and that $\lim_{\epsilon \downarrow 0} O_j(x+\mathrm{i}\epsilon)\in \T$ in the second. Therefore, it follows from equation \eqref{For:45809} that the set 
$P_{h_1}\cup P_{h_2} \subset \R$ has full measure. 
\par\smallskip
Now, since $g \in \mathcal{B}(h_1)\cap \mathcal{B}(h_2)$ we conclude using Proposition \ref{Prop:BoundaryBehaviour} that 
\begin{align*}
    \lim_{\epsilon \downarrow 0} g(x+\mathrm{i}\epsilon)
    =\lim_{\epsilon \downarrow 0} g(x-\mathrm{i}\epsilon) 
    \text{ for $\lambda$ almost all $x \in P_{h_1}\cup P_{h_2}$ or for $\lambda$ almost all $x \in  \R$}.
\end{align*} 
This shows that functions $g_{|\C^+}$ and $g_{|\C^-}$ are pseudocontinuations of each other.
\color{black}
\par\medskip
    \emph{2. The function $g_{|\C^+}$ is an element of $\mathcal{H}^2(\C^+)$}
 \newline
  It is sufficient to prove that either 
   $O_1$ and $O_2$ is extreme (Proposition \ref{Prop:Rovnak}). Since $O_2$ is non-constant, there exists a set  $E\subset \R$ such that
   \begin{align*}
        \lambda(E)>0 \quad  \text{ and } \quad \log(|O_2(x)|)<0 \quad \text{ for $\lambda$ almost all $x\in E$}.
   \end{align*}
    Consequently, $\log(|O_1(x)|)=0$, or equivalently $|O_1(x)|=1$, holds $\lambda$ a.e. on $E$. Thus  $O_1$ is extreme. 
 \par\medskip
    \emph{3. The function $g_{|\C^-}$ is an element of $\mathcal{N}^+(\C^-)$}.
 \newline
 Let $q_i=C^{-1}(h_i)$ be the associated Herglotz-Nevanlinna function. Moreover, recall that $\mathcal{B}(h_i)$  is isomorphic to  $\mathcal{L}(q_i)$ via the multiplication
 \begin{align*}
    \mathcal{B}(h_i)=\frac{1+h_i}{\sqrt{2}} \cdot \mathcal{L}(q).
\end{align*}
We have seen in Proposition \ref{Prop:RealHnfunctions} that $\mathcal{L}(q_i)\subset \mathcal{N}^+(\C \setminus \R)$. Moreover, it also holds that $(h_i)_{|\C^-} \in  \mathcal{N}^+(\C^- )$, since 
\begin{align*}
    (h_i)_{|\C^-}=\frac{1}{SR(O_i)}
\end{align*}
is an outer function. This means that $g_{|\C^-}$ is an element of $\mathcal{N}^+(\C^-)$ because  $\mathcal{N}^+(\C^-)$ is closed under multiplications.
\par\smallskip
Combining the three statements above with Proposition \ref{Prop:RegRes} yields that $g=0$. 
\end{proof}
The final argument missing for the proof of Theorem \ref{basic1.5} is  the following
rather basic observation:
\begin{lemma} \label{Lem:Div}
    Let $g \in \mathcal{N}(\C^+) \setminus \mathcal{N}^+(\C^+)$. Furthermore, assume that there are two inner functions $W_1$ and $W_2$ such that $g \cdot W_1$ and $g \cdot W_2$ are elements of $\mathcal{N}^+(\C^+)$. Then there exists a non-trivial inner function which divides both $W_1$ and $W_2$.
\end{lemma}
\begin{proof}
We can write the function $g$ as
\begin{align*}
    g= \frac{V_1}{V_2} \cdot F
\end{align*}
where $V_1$ and $V_2$ are inner functions such that $V_2$ is non trivial and $V_1$ and $V_2$ are relatively prime, and $F$ is an outer function. In order for $g \cdot W_1$ or $g \cdot W_2$ to be an element of $\mathcal{N}^+(\C^+)$, the inner function $W_1$ or $W_2$ has to cancel out $V_2$. This means that $V_2$ divides both $W_1$ and $W_2$. 
\end{proof}
\begin{proof}[Proof of Theorem \ref{basic1.5}]
First assume that there exists a non-trivial  function $m \in \mathcal{S}_0$ that divides both $h_1$ and $h_2$. Let  $m_1, m_2 \in \mathcal{S}_0$ be such that $m \cdot m_i=h_i$.  It is straightforward to verify that
\begin{align*}
    s_{h_i}(\zeta,w) = s_{m}(\zeta,w)+m(\zeta) \cdot s_{m_i}(\zeta,w) \cdot\overline{ m(w)}.
\end{align*}
Consequently, we see that the kernel $s_{h_i}$ is the sum of two positive definite kernels. In this case, it is well known \cite[Part 1, Section 6]{7069433019500501}, that 
\begin{align*}
    \mathcal{B}(h_i)= \mathcal{B}(m)+ m \cdot \mathcal{B}(m_i),
\end{align*}
where the sum is not necessarily direct. Thus, $\mathcal{B}(m) \subset \mathcal{B}(h_1)\cap \mathcal{B}(h_2)$ holds, which concludes the proof of this direction.
\par\medskip
Conversely, let $h_1$ and $h_2$ be elements of $\mathcal{S}_0$ such that  $h_1$ and $h_2$ are relatively prime. We can decompose  $(h_i)_{|\C^+}$ as $(h_i)_{|\C^+}=V_i \cdot O_i$, where $V_i$ is an inner and $O_i$ is an outer function. 
The same reasoning as above yields 
\begin{align*}
    \mathcal{B}(h_i)=\mathcal{B}(o_i)+o_i  \mathcal{B}(v_i),
\end{align*}
where $o_i$ and $v_i$ denote the functions in $\mathcal{S}_0$ generated by  $O_i$ and $V_i$ respectively. 
\par\smallskip
Let
$g \in \mathcal{B}(h_1)\cap \mathcal{B}(h_2)$. Then we can write $g$ as a sum 
\begin{align*}
    g=f_1 + o_1 \cdot k_1=f_2 + o_2 \cdot k_2 , \quad \text{ with } f_i \in \mathcal{B}(o_i), \ k_i \in \mathcal{B}(v_i).
\end{align*}
We have seen in the proof of Lemma \ref{Lem:IntOuter} that $(f_i)_{|\C^-}\in \mathcal{N}^+(\C^-)$ and that $(o_i)_{|\C^-} \in \mathcal{N}^+(\C^-)$. From Proposition \ref{Schur:Basic2} we know that $(k_i)_{|\C^-}$ is either trivial or $(k_i)_{|\C^-} \notin \mathcal{N}^+(\C^-)$. In the latter case, it then holds that 
\begin{align*}
(k_i)_{|\C^-} \cdot SR(V_i) \in \mathcal{N}^+(\C^-), \  i \in \{1,2\}.
\end{align*}
First assume that $g_{|\C^-} \in \mathcal{N}^+(\C^-)$. Then it follows that both $(k_1)_{|\C^-}$ and $(k_2)_{|\C^-}$ have to be trivial, which in turn implies that $g \in \mathcal{B}(o_1)\cap \mathcal{B}(o_2) $. However, their intersection is trivial by Lemma \ref{Lem:IntOuter}.
\par\smallskip
Now assume that $g_{|\C^-} \notin \mathcal{N}^+(\C^-)$. In this case,  we observe that both $g_{|\C^-} \cdot SR(V_1)$ and 
$g_{|\C^-} \cdot SR(V_2)$ are elements of $\mathcal{N}^+(\C^-)$. Consequently, it follows that the inner functions $SR(V_1)$ and $SR(V_2)$ and in turn also $V_1$ and $V_2$ have a nontrivial common divisor by Lemma \ref{Lem:Div}. This contradicts our assumption that  $h_1$ and $h_2$ are relatively prime.   
\end{proof}
\color{black}
\subsection{The behaviour on the boundary}
Now we give a proof Lemma  \ref{BoundaryBehaviour}, which is restated for convenience here:
\begin{proposition}\label{lem:boundary}
Let $h_1,h_2 \in \mathcal{S}_0$ such that they are relatively prime. Then the function 
\begin{align*}
    f= \frac{h_1}{h_2}
\end{align*}
can be analytically continued through an open interval $(a,b)\subset \R$ if and only if both $h_1$ and $h_2$ can be analytically continued through $(a,b)$. 
\end{proposition}

\begin{proof}
First, assume that  both $h_1$ and $h_2$ can be analytically continued through $(a,b)$. Then $h_1$ and $h_2$ have absolute value 1 on $(a,b)$, and are in particular not zero. Consequently, 
the function $f$ can be analytically continued through $(a,b)$ as well. 
\par\smallskip
Conversely, let $f$ be analytically continuable through $(a,b)$. Then it follows from the definition of $h_1$ and $h_2$ that  
\begin{align*}
\overline{f(\zeta)}=\frac{1}{f(\overline{\zeta})} \quad \forall \zeta \in \C \setminus \R,
\end{align*}
and we conclude that $|f(w)|=1$ on $(a,b)$. Since  $h_1$ and $h_2$ are relatively prime, it follows that $h_1$ and $h_2$ can be written as 
\begin{align*}
    (h_1)_{|\C^+}(\zeta)&= B_1(\zeta)  \exp\left( \frac{-\mathrm{i}}{\pi} \int_{\R}   \left( \frac{1}{t-\zeta}-\frac{t}{1+t^2} \right)  \mathrm{d}\mu_1 (t)  \right)    \exp(\mathrm{i} \alpha_1 \zeta)  \\ 
    (h_2)_{|\C^+}(\zeta)&= B_2(\zeta)  \exp\left(   \frac{-\mathrm{i}}{\pi} \int_\R \left( \frac{1}{t-\zeta}-\frac{t}{1+t^2} \right)   \mathrm{d}  \mu_2 (t)  \right)    \exp(\mathrm{i} \alpha_2 \zeta) 
\end{align*} 
where $B_1$ and $B_2$ are Blaschke products without common zeros, $\mathrm{d}\mu_1$ and $\mathrm{d}\mu_2$ are negative measures which are mutually singular, and at least one of $\alpha_1$ and $\alpha_2$ is zero. Since a function of the form $\exp(\mathrm{i} \alpha \zeta) $ is analytic on $\C$ anyway, we can assume without loss of generality that $\alpha_1=\alpha_2=0$.
\par\smallskip
Assume, for sake of contradiction, that there exists an element $ \zeta \in (a,b)$ such that $\zeta$ is an accumulation point for the zeros of either $B_1$ or $B_2$. Then, since $B_1$ and $B_2$ have no common zeros,  there exists a sequence of zeros or poles of $f$ in $\C^+$ converging to $ \zeta$. This  contradicts that $f$ is analytic on $(a,b)$. Therefore,  $\zeta$ cannot be such an accumulation point, and we conclude that the functions in $\mathcal{S}_0$ generated by $B_1$ and $B_2$ are analytic on $(a,b)$ and do not vanish there 
(Proposition \ref{Prop:AnaExt}).
\par\smallskip
Therefore, we can assume without loss of generality that $B_1=B_2=1$, since 
$f$ is analytic on $(a,b)$ if and only if $f \cdot \frac{B_2}{B_1}$ is analytic on $(a,b)$. In other words, we assume that the function $f$ is analytic on $\C^+ \cup (a,b)$ and does not vanish there. This allows us to consider  
\begin{align*}
q(\zeta)= -\mathrm{i}  \log(f)(\zeta)=  \frac{1}{\pi}  \int_\R \left( \frac{1}{t-\zeta}-\frac{t}{1+t^2} \right)\mathrm{d} (\mu_1 - \mu_2) (t).
\end{align*}
Note that $q$ is also analytic on $\C^+\cup (a,b)$. It is known \cite[Corollary 4.3]{luger2019quasiherglotz}, that the measure $\mu_1-  \mu_2$
is given by the weak-$(*)$ limit of the densities
\begin{align*}
 p_y(t)=\mathrm{Im}(q)(t+iy) 
 = \log{|f(t+iy)|} , \quad y>0.
\end{align*}
Since $|f(w)|=1$ on $(a,b)$ and $f$ is analytic on $(a,b)$, it follows that the densities $p_y$ go to zero locally uniformly on $(a,b)$.  Therefore, we conclude that  $(\mu_1-  \mu_2) \cdot \mathbbm{1}_{(a,b)}$
is the trivial measure. 
In other words, it holds that 
\begin{align*}
    \mu_1 \cdot \mathbbm{1}_{(a,b)}= \mu_2 \cdot \mathbbm{1}_{(a,b)}.
\end{align*}
In summary, $\mu_1 \cdot \mathbbm{1}_{(a,b)}$ and $\mu_2 \cdot \mathbbm{1}_{(a,b)}$ are both absolutely continuous with respect to each other (they are equal) and  mutually singular by construction. 
This means that $\mu_j \cdot \mathbbm{1}_{(a,b)}=0$ for $j=1,2$ (see e.g. 
\cite[Chapter 2: Exercise 2.13.2]{Friedman1982fom}). 
\par\smallskip
Finally, since $\mu_j \cdot \mathbbm{1}_{(a,b)}$ are the trivial measures, we conclude that the Herglotz-Nevanlinna functions 
\begin{align*}
    \frac{1}{\pi} \int_{\R}   \left( \frac{1}{t-\zeta}-\frac{t}{1+t^2} \right)  \mathrm{d}\mu_j (t)
\end{align*}
are analytic on the interval $(a,b)$ \cite[Theorem 3.20]{teschl2009mathematical}. This means that $h_1$ and $h_2$ are analytic on $(a,b)$ since they are given by the following formula: 
\begin{align*}
    h_j(\zeta)&=   \exp\left( \frac{-\mathrm{i}}{\pi} \int_{\R}   \left( \frac{1}{t-\zeta}-\frac{t}{1+t^2} \right)  \mathrm{d}\mu_j (t)  \right).
\end{align*}
Recall that we had treated $B_j$ and $\exp(\mathrm{i} \alpha_j \zeta)$ separately and were therefore able to assume without loss of generality that they are trivial.
\color{black}
\end{proof}

\section{Model relations on Krein spaces}
The construction of minimal realizations presented in Section \ref{Sec:Construction} 
defines an interesting class of self-adjoint relations on Krein spaces, which we will call model relations. In this section we describe  their  structure  in  detail. Our first Proposition is  a simple transformation result, which  will turn out  useful at a later stage.  
\begin{proposition}
Let $f \in \mathcal{N}_{sym}$ and $(\mathcal{L}(f),A_f,\phi_f)$ be the realization constructed in Section \ref{Sec:Construction}, where 
\begin{align*}
    \mathcal{L}(f)  = \frac{\sqrt{2}}{h_1+h_2} \cdot (\mathcal{B}(h_1) \oplus (- \mathcal{B}(h_2))
\end{align*}
for a pair of functions $h_1,h_2 \in \mathcal{S}_0$ which are relatively prime. Moreover, let 
\begin{align*}
    T \vcentcolon (\mathcal{B}(h_1) \oplus (- \mathcal{B}(h_2)) \rightarrow \mathcal{L}(f), \quad 
    g \mapsto  \frac{\sqrt{2}}{h_1+h_2} \cdot g
\end{align*}
denote the isometric isomorphism. Then the relation  $\tilde{A}_f = T^{-1} A_f T$ has the following form: 
\begin{align*}
     \tilde{A}_f = \left\{(h,g) \in (\mathcal{B}(h_1)\oplus (-\mathcal{B}(h_2)))^2 \ | \ \exists c \in \C \vcentcolon   g(\zeta)-\zeta  h(\zeta) \equiv c  \left( h_1(\zeta) + h_2(\zeta)  \right)   \right\}.
\end{align*}

\end{proposition}
The advantage of considering $A_f$ on $\mathcal{B}(h_1)\oplus -\mathcal{B}(h_2)$ is that in this case we are given an explicit fundamental decomposition consisting of two well-studied function spaces. 
\begin{proof}
Recall that $A_f$ on $\mathcal{L}(q)$ is given by 
\begin{align*}
A_f= \{ (h,g)\in \mathcal{L}(q) \times \mathcal{L}(q) \vcentcolon \exists c \in \C \vcentcolon g(\zeta)-\zeta  h(\zeta)\equiv c \}.
\end{align*} 
It is then straightforward to verify that $\tilde{A}_f$ has the claimed form. 
\end{proof} 
The theory developed in \cite{article55} provides a suitable framework to describe the structure of model relations. Following their approach, we introduce partially fundamentally reducible relations:
\begin{definition}
 Let $A$ be a self-adjoint relation in a Krein space $\mathcal{K}$ satisfying $\varrho(A)\neq \emptyset$. Then 
 \begin{itemize}
 \item the relation $A$ is called  fundamentally reducible if there exists a fundamental decomposition $\mathcal{K}=\mathcal{K}_+ [+] \mathcal{K}_-$ such that  $A$ splits into an orthogonal sum $A=A_1 \oplus A_2$ with respect to this decomposition.
 \item the relation $A$ is called partially fundamentally reducible if 
    there exists a fundamental decomposition $\mathcal{K}=\mathcal{K}_+ [+] \mathcal{K}_-$ such that the relations
\begin{align*}
    S_+ \vcentcolon = A \cap (\mathcal{K}_+ \times \mathcal{K}_+) 
    \quad \text{ and } \quad 
    S_- \vcentcolon = A \cap (\mathcal{K}_- \times \mathcal{K}_-) 
\end{align*}
are closed, simple symmetric operators on the Hilbert spaces $\mathcal{K}_+$ and $-\mathcal{K}_-$ with deficiency index $(1,1)$.
\end{itemize}
\end{definition}
The structure of  fundamentally reducible relations is relatively easy to analyse, since such a  relation is also self-adjoint with respect to the induced Hilbert space inner product. As the name suggests, partially fundamentally reducible relations are relatively close to fundamentally reducible ones. This  will be specified  at a later stage. First, we  prove that model relations are indeed of that form:
\begin{theorem}\label{Thrm:orthcoup}
    Let $A_f$ be a model relation corresponding to a function $f$ such that neither $f$ nor $-f$   is  a Herglotz-Nevanlinna function. Then $A_f$ is partially fundamentally reducible.   
\end{theorem}
\begin{proof}
    We  assume without loss of generality that $A_f$ is given by 
    \begin{align*}
    A_f = \left\{(h,g) \in (\mathcal{B}(h_1)\oplus -\mathcal{B}(h_2))^2 | \ \exists c \in \C \vcentcolon   g(\zeta)-\zeta  h(\zeta) \equiv c  \left(    h_1(\zeta) + h_2(\zeta)  \right)   \right\}.
\end{align*}
Consider an element $(h,g) \in A_f \cap (\mathcal{B}(h_1) \times \mathcal{B}(h_1))$. Then there exists a constant $c \in \C$ such that 
\begin{align*}
    g(\zeta)-\zeta  h(\zeta) \equiv c  \left(  h_1(\zeta) + h_2(\zeta)  \right).
\end{align*}
Applying $D_w$ to the left-hand side, for a suitable 
$w$, and using its multiplication identity, we obtain
\begin{align*}
    D_w(g-\zeta \cdot f)=
    D_w(g)- D_w(\zeta \cdot f) = D_w(g)- w  D_w(f)-f \cdot D_w(\zeta)= D_w(g)-w  D_w(f)-f .
\end{align*}
This is an element of $\mathcal{B}(h_1)$, because both $f$ and $g$ are and $D_w$ acts boundedly on $\mathcal{B}(h_1)$. Consequently, it follows that also
\begin{align*}
     D_w(c  \left(  h_1 + h_2  \right))= c  (  D_w( h_1) +D_w( h_2))\in \mathcal{B}(h_1).
\end{align*}
However, we know that $D_w( h_i) \in \mathcal{B}(h_i)$. This  means that $D_w(c  \left(  h_1 + h_2  \right))$ is either the zero function if $c=0$, or not an element of $\mathcal{B}(h_1)$. Consequently, we arrive at
\begin{align*}
A_f\cap (\mathcal{B}(h_1) \times \mathcal{B}(h_1)) = \left\{(h,g) \in \mathcal{B}(h_1) \times \mathcal{B}(h_1) \ |  g(\zeta)-\zeta  h(\zeta) \equiv 0 \right\}.
\end{align*}
This means that  $A_f\cap (\mathcal{B}(h_1) \times \mathcal{B}(h_1))$ is the multiplication operator with the independent variable on $\mathcal{B}(h_1)$, which is a simple, symmetric operator with deficiency index $(1,1)$. The second relation $A_f\cap (\mathcal{B}(h_2) \times \mathcal{B}(h_2))$ is treated analogously. 
\end{proof}
The next theorem is, in principle, a special case of \cite[Theorem 4.7]{article55}. The main difference is that the symmetric operators we consider here are not necessarily densely defined, but that is an assumption that is not crucial to their argument. In our specific situation, this result can be proven in a more elementary manner using the multiplication identity of $D_w$. 
\begin{theorem}\label{Thrm:FundRed}
Let $f\in \mathcal{N}_{sym}$ such that neither $f$ nor $-f$ is a Herglotz-Nevanlinna function. Moreover, consider the realization 
$(\mathcal{L}(f),A_f,\phi_f)$ constructed in Section \ref{Sec:Construction}, where 
\begin{align*}
    \mathcal{L}(f)  = \frac{\sqrt{2}}{h_1+h_2} \cdot (\mathcal{B}(h_1) \oplus (- \mathcal{B}(h_2))
\end{align*}
for a pair of functions $h_1,h_2 \in \mathcal{S}_0$ which are relatively prime. Finally, set 
\begin{gather*}
    q_1(\zeta)= \mathrm{i} \frac{ 1-h_1(\zeta)}{1+h_1(\zeta)}
    \quad 
k_1(\zeta)\vcentcolon =\frac{1+h_1(\zeta)}{h_1(\zeta)+h_2(\zeta)} \\
q_2(\zeta)=-\mathrm{i} \frac{1+h_2(\zeta)}{ 1-h_2(\zeta)}
\quad 
k_2(\zeta)\vcentcolon = \frac{1-h_2(\zeta)}{h_1(\zeta)+h_2(\zeta)} \\
\phi(w)=k_1 \cdot D_w(q_1) +k_2 \cdot D_w(q_2).
\end{gather*}
Then $\phi(w)\in \mathcal{L}(f)$ for every $w \in \varrho(A_f)\cap \C \setminus \R$, and  there exists a fundamentally reducible relation $\tilde{A}$
such that 
\begin{align}\label{ForUnterschied}
    (A_f-w)^{-1}=(\tilde{A}-w)^{-1} - \frac{[\ \cdot \ ,\phi(\overline{w})]}{ q_1(w)+q_2(w)} \cdot \phi(w) \quad \forall w \in \varrho(A)\cap \varrho(\tilde{A}).
\end{align}
\end{theorem}
\begin{proof}
First, we note that the functions $q_1$ and $q_2$ are related to $h_1$ and $h_2$ via the Cayley transform. Specifically, we have:
\begin{align*}
q_1=C(h_1) \text{ and } \frac{1}{q_2}=C(h_2).
\end{align*}
The functions $C(h_i), \ \mathrm{i} \in \{1,2\}$ are Herglotz-Nevanlinna functions. This means that both $q_1$ and $-q_2$ are Herglotz-Nevanlinna functions as well. Moreover, we can represent $\mathcal{L}(f)$ as 
\begin{align*}
    \mathcal{L}(f)= \frac{\sqrt{2}}{h_1+h_2} \cdot (\mathcal{B}(h_1) \oplus (- \mathcal{B}(h_2))=k_1 \cdot \mathcal{L}(q_1) \oplus \bigg(-(k_2 \cdot \mathcal{L}(-q_2))\bigg).
\end{align*}
This representation is a consequence of Theorem \ref{Thrm:BhLq}, according to which it holds that 
\begin{align}
\begin{split}\label{For:Conjproof}
    \frac{\sqrt{2}}{h_1+h_2} \cdot \mathcal{B}(h_1)&= \frac{1+h_1  }{h_1+h_2} \cdot \mathcal{L}(q_1) =k_1 \cdot \mathcal{L}(q_1)\\
    \frac{\sqrt{2}}{h_1+h_2} \cdot \mathcal{B}(h_2)&= \frac{1+h_2  }{h_1+h_2} \cdot \mathcal{L}\left(\frac{1}{q_2}\right)
    =\frac{1+h_2  }{h_1+h_2} \cdot \frac{1}{q_2} \cdot \mathcal{L}(-q_2)=
    k_2 \cdot \mathcal{L}(-q_2).
\end{split}
\end{align}
Here, we have used  the identity $\mathcal{L}\left(\frac{1}{q_2}\right)=\frac{1}{q_2} \cdot \mathcal{L}(-q_2)$, which was established in Section \ref{GenBhspaces}. 
\par\smallskip
Now consider the self-adjoint relations $A_1$ and $A_2$ on $\mathcal{L}(q_1)$ and $\mathcal{L}(-q_2)$ given by $(A_i-w)^{-1}=D_w$.  We define the desired fundamentally reducible self-adjoint relation $\tilde{A}\subset \mathcal{L}(f)\times \mathcal{L}(f)$ as follows:
\begin{align*}
    \tilde{A}=     \left( k_1 \cdot A_1 \cdot \frac{1}{k_1} \right) \oplus  \left(k_2 \cdot  A_2 \cdot \frac{1}{k_2}\right).
\end{align*}
Here, the multiplication by $k_1$ stands for the isomorphism from $\mathcal{L}(q_1)$ to $k_1 \cdot \mathcal{L}(q_1)$ and all other multiplications are to be understood accordingly. 
Next, let $k_1 \cdot g$ belong to $k_1 \cdot \mathcal{L}(q_1)$. We  compute 
\begin{align*}
    ((A_f-w)^{-1}-(\tilde{A}-w)^{-1})(k_1 \cdot g)&=
    D_w(k_1 \cdot g)-\left(k_1 \cdot D_w \cdot \frac{1}{k_1}\right)(k_1 \cdot g)\\
    &=k_1 \cdot D_w(g)+g(w) \cdot D_w(k_1)-k_1 \cdot D_w(g) =g(w) \cdot D_w(k_1).
\end{align*}
We also observe that
\begin{align*}
    [ k_1 \cdot g ,\phi(\overline{w})]_{\mathcal{L}(f)} = [ k_1 \cdot g ,k_1 \cdot N_{q_1}(\cdot,w)]_{\mathcal{L}(f)} = [ g , N_{q_1}(\cdot,w)]_{\mathcal{L}(q_1)}=g(w).
\end{align*}
 Therefore, to establish \eqref{ForUnterschied}, we only need to demonstrate that
\begin{align*}
    (q_1(w)+q_2(w)) D_w(k_1)  +\phi(w)=0.
\end{align*}
To this end, we first note that $k_1=1+k_2$, which means that $D_w(k_1)=D_w(k_2)$. Consequently, it follows that
\begin{align*}
    (q_1(w)+q_2(w)) D_w(k_1) +\phi(w) &=
    q_1(w) D_w(k_1)   + k_1 \cdot D_w(q_1) +q_2(w) D_w(k_2)+ k_2 \cdot D_w(q_2) \\
    &=D_w(k_1\cdot q_1)+D_w(k_2\cdot q_2)=D_w(k_1\cdot q_1+k_2\cdot q_2)=0,
\end{align*}
where we have used the multiplication identity of $D_w$ and 
\begin{align*}
    k_1\cdot q_1+k_2\cdot q_2=-\mathrm{i}.
\end{align*}
 A similar calculation shows that \eqref{ForUnterschied} is also valid for functions in the negative part of $\mathcal{L}(f)$, which is  $ k_2 \cdot \mathcal{L}(-q_2)$. This concludes our proof. 
\end{proof}
A similar  idea also allows us to prove the  ``converse'' of Theorem \ref{Thrm:orthcoup}.
\color{black}
\begin{theorem}\label{MainConverse}
 Let 
 $A$ be a  partially fundamentally reducible relation on $\mathcal{K}$. Then $\sigma(A)\cap (\C \setminus \R)$ is a discrete set in $\C \setminus \R$. Moreover, if $(\mathcal{K},A,v)$ is a realization of a function $q$ with base point $\zeta_0\in \varrho(A)$, then $q$ is an element of $\mathcal{N}_{sym}$.
\end{theorem}
\begin{proof} 
Let 
 $A$ be a  partially fundamentally reducible relation on $\mathcal{K}$. Then there exists a fundamental decomposition $\mathcal{K}=\mathcal{K}_+ [+] \mathcal{K}_-$ such that the relations
\begin{align*}
    S_+ \vcentcolon = A \cap (\mathcal{K}_+ \times \mathcal{K}_+) 
    \quad \text{ and } \quad 
    S_- \vcentcolon = A \cap (\mathcal{K}_- \times \mathcal{K}_-) 
\end{align*}
are closed, simple symmetric operators on the Hilbert spaces $\mathcal{K}_+$ and $-\mathcal{K}_-$ with deficiency index $(1,1)$. 
\par\smallskip
Let  $(\Gamma_0^+,\Gamma_1^+, \C)$ and $(\Gamma_0^-,\Gamma_1^-, \C)$ be arbitrary but fixed boundary triples   for  $S_+$ and $S_-$, which exist by  Proposition \ref{Prop:Symop2}. Moreover, let $q_+$ and $q_-$ denote the Weyl-functions, and $y_+$ and $y_-$ denote the Weyl-solutions corresponding to these boundary triples. Note that $q_+$ and $q_-$ are Herglotz-Nevanlinna functions. We also define self-adjoint relations $A_+$ and $A_-$ by
\begin{align*}
    A_+=\mathrm{ker}(\Gamma_0^+) \quad \text{ and } A_-=\mathrm{ker}(\Gamma_0^-).
\end{align*}
Now consider the closed and symmetric operator $S_+ \oplus S_-$ on the Krein space $\mathcal{K}$.  It is clear by construction that both $A$ and $A_+\oplus A_-$ are  self-adjoint extension of $S_+ \oplus S_-$. Moreover, the triple   
\begin{align*}
    (\Gamma_0^+ \oplus (-\Gamma_0^-), \Gamma_1^+ \oplus \Gamma_1^-, \C^2)
\end{align*}
is the boundary triple for $S_+ \oplus S_-$ satisfying  $\mathrm{ker}((\Gamma_0^+ \oplus (-\Gamma_0^-))=A_0\oplus A_1$.
The respective Weyl solution $\gamma$ is given by $\gamma=\gamma_+\oplus \gamma_-$, and the respective Weyl function by
\begin{align*}
m(\zeta) = \begin{pmatrix}
    q_+(\zeta) &0 \\
    0 &-q_-(\zeta)
\end{pmatrix}.
\end{align*}
This allows us to use Krein`s resolvent formalism \cite[Corollary 2.4]{RN07186998219990101}. More precisely, there  exists two $2 \times 2$ matrices $G=(g_{ij})$ and $H=(h_{ij})$ such that 
\begin{align*}
    \zeta \in \varrho(A) \Leftrightarrow 0 \in \varrho(G-m(\zeta) \cdot H) \Leftrightarrow \mathrm{det}(G-m(\zeta) \cdot H)\neq 0.
\end{align*}
We can explicitly compute the determinant as 
\begin{align*}
    \mathrm{det}(G-m(\zeta) \cdot H)&=
    \mathrm{det}\left( 
    \begin{pmatrix}
        g_{11}-q_+(\zeta)  h_{11} &
        g_{12}-q_+(\zeta)  h_{12} \\
        g_{21}+q_-(\zeta)  h_{21}
        &
        g_{22}+q_-(\zeta)  h_{22}
    \end{pmatrix} \right) \\
    &=(g_{11}-q_+(\zeta)  h_{11}) (g_{22}+q_-(\zeta)  h_{22}) -  (g_{21}+q_-(\zeta)  h_{21}) (g_{12}-q_+(\zeta)  h_{12}).
\end{align*}
This is an analytic function in $\mathcal{N}(\C\setminus \R)$, which means that $\sigma(A)\cap (\C \setminus \R)$ is a discrete set in $\C \setminus \R$. Finally, Krein's resolvent formula takes the form 
\begin{align*}
    (A-\zeta)^{-1}= (A_0 -\zeta)^{-1} + \gamma(\zeta) ( H (G-m(\zeta) \cdot H)^{-1}) \gamma^+(\overline{\zeta}) \quad \zeta \in  \varrho(A), 
\end{align*}
which allows us to calculate
\begin{align}
\begin{split}\label{For:Modelop1}
    \left[\left(I+(\zeta-\zeta_0)(A-\zeta)^{-1}\right) v, v\right]_{\mathcal K} &= [(I+(\zeta-\zeta_0)(A_0 -\zeta)^{-1} v, v]_{\mathcal K} \\
    &+ [ \gamma(\zeta) ( H (G-m(\zeta)  H)^{-1}) \gamma^+(\overline{\zeta})  v,v]_{\mathcal K}.
\end{split}
\end{align}
In what follows, our strategy is to show that all occurring functions are of bounded type. 
Let $v=v_++v_-$ be the decomposition of $v$ with respect to the fundamental decomposition $\mathcal{K}_+ [+] \mathcal{K}_-$. Then 
the first summand simplifies to 
\begin{align*}
    [(I+(\zeta-\zeta_0)(A_0 -\zeta)^{-1} v, v]_{\mathcal K} 
    &=
    [(I+(\zeta-\zeta_0)(A_+ -\zeta)^{-1} v_+, v_+]_{\mathcal{K}_+} \\
    &+[(I+(\zeta-\zeta_0)(A_- -\zeta)^{-1} v_-, v_-]_{\mathcal{K}_-} = s_1-s_2,
\end{align*}
where $s_1$ and $s_2$ are Herglotz-Nevanlinna functions and therefore elements of $\mathcal{N}(\C \setminus \R)$. Here, we have used that $A_0$ decomposes into an orthogonal sum $A_0=A_+\oplus A_-$ with respect to the fundamental decomposition $\mathcal{K}=\mathcal{K}_+ [+] \mathcal{K}_-$.
\par\smallskip
Now we turn our attention to the second summand in \eqref{For:Modelop1}. It was shown in  Proposition \ref{Prop:Symop2} that
\begin{align*}
    \zeta \mapsto (\gamma_+(\overline{\zeta}))^+(v_+)&=
    [v_+,\gamma_+(\overline{\zeta})]_{\mathcal{K}_+} \in \mathcal{L}(q_1)\subset \mathcal{N}(\C \setminus \R), \\
    \zeta \mapsto  (\gamma_-(\overline{\zeta}))^+(v_-)&=
    [v_-,\gamma_-(\overline{\zeta})]_{\mathcal{K}_-}
    \in \mathcal{L}(-q_2)\subset \mathcal{N}(\C \setminus \R).
\end{align*}
Moreover, we compute $H (G-m(\zeta) \cdot H)^{-1}$  using the formula for the inverse of a $2 \times 2$ matrix as
\begin{align*}
    H (G-m(\zeta) \cdot H)^{-1} &= 
    \mathrm{det}(G-m(\zeta) \cdot H) \cdot 
    H \cdot
    \begin{pmatrix}
        g_{22}+q_-(\zeta)  h_{22} &
        -g_{21}-q_-(\zeta)  h_{21} \\
        -g_{12}+q_+(\zeta)  h_{12}
        &
        g_{11}-q_+(\zeta)  h_{11}
    \end{pmatrix} \\
    &\vcentcolon=\begin{pmatrix}
          l_{11}(\zeta) &l_{12}(\zeta)\\
          l_{21}(\zeta) &l_{22}(\zeta)
      \end{pmatrix}.
\end{align*}
Since $\mathrm{det}(G-m(\zeta) \cdot H)$ is an element of $\mathcal{N}(\C \setminus \R)$ and the entries of $H$ are constants, it follows that the entries
$(l_{ij}(\zeta))$ are also elements of $\mathcal{N}(\C \setminus \R)$.  We conclude that
\begin{align*}
     \gamma(\zeta) ( H (G-m(\zeta) \cdot H)^{-1}) \gamma^+(\overline{\zeta})  v &= 
      \gamma(\zeta) \begin{pmatrix}
          l_{11}(\zeta) &l_{12}(\zeta)\\
          l_{21}(\zeta) &l_{22}(\zeta)
      \end{pmatrix}  \cdot 
      \begin{pmatrix}
          [v_+,\gamma_+(\overline{\zeta})] \\
          [v_-,\gamma_-(\overline{\zeta})]
      \end{pmatrix} \\
      &=\bigg(l_{11}(\zeta)  [v_+,\gamma_+(\overline{\zeta})]+l_{12}(\zeta)  [v_-,\gamma_-(\overline{\zeta})]\bigg)  \gamma_+(\zeta) \\
      &+\bigg(l_{21}(\zeta)  [v_+,\gamma_+(\overline{\zeta})]+l_{22}(\zeta)  [v_-,\gamma_-(\overline{\zeta})]\bigg)  \gamma_-(\zeta) \\
      &= s_3(\zeta)  \gamma_+(\zeta) + s_4(\zeta)  \gamma_-(\zeta),
\end{align*}
where $s_3$ and $s_4$ are functions of the class $\mathcal{N}(\C\setminus \R)$. Consequently, it follows that
\begin{align*}
    [ \gamma(\zeta) ( H (G-m(\zeta) \cdot H)^{-1}) \gamma^+(\zeta)v,v]_{\mathcal K} 
    &=
    [  s_3(\zeta)  \gamma_+(\zeta) + s_4(\zeta)  \gamma_-(\zeta),v]_{\mathcal K} \\
    &= s_3(\zeta)  \overline{[v_+, \gamma_+(\zeta) ]}_{\mathcal{K}_+}+ 
    s_4(\zeta)  \overline{[v_-,\gamma_-(\zeta)]}_{\mathcal{K}_-} \in \mathcal{N}(\C\setminus \R),
\end{align*}
where  we have used that $\overline{[v_+,\gamma_+(\zeta)]}_{\mathcal{K}_+}$  and $\overline{[v_-,\gamma_-(\zeta)]}_{\mathcal{K}_-}$ belong to $\mathcal{N}(\C\setminus \R)$.
\par\smallskip
Originally, we had assumed that $(\mathcal{K},A,v)$ is a realization of a function $q$, i.e. that
\begin{align*}
    q(\zeta)= \overline{q(\zeta_0)}+(\zeta- \overline{\zeta_0})  \left[\left(I+(\zeta-\zeta_0)(A-\zeta)^{-1}\right) v, v\right]_{\mathcal K} \quad \forall \zeta \in \varrho(A).
\end{align*}
Our preceding analysis has shown that  the inner product coincides with a function in $ \mathcal{N}(\C \setminus \R)$ on $\varrho(A)$,  which means that $q\in \mathcal{N}(\C \setminus \R)$. This concludes the proof. 
\end{proof}
We can summarize our results in the following theorem: 
\begin{theorem}
Let $f$ be a meromorphic function on $\C \setminus \R$ and let
\begin{align*}
    U \vcentcolon = (\C \setminus \R) \setminus \{w \in \C \setminus \R \vcentcolon \ w \ \text{is a pole of} \ f \}
\end{align*}
denote the domain of analyticity. Then $f\in \mathcal{N}_{sym}$ if and only if there  exists a Krein space $\mathcal{K}$ and a partially fundamentally reducible relation $A$ in $\mathcal{K}$ such that 
   \begin{gather*}
U \subset\varrho(A) \quad \text{ and} \quad 
f(\zeta)= \overline{f(\zeta_0)}+(\zeta- \overline{\zeta_0})  \left[\left(I+(\zeta-\zeta_0)(A-\zeta)^{-1}\right) v, v\right]_{\mathcal K} \quad \forall \zeta \in U  \\
\overline{\mathrm{span}} \{ \left(I+(\zeta-\zeta_0)(A-\zeta)^{-1}\right) v  \vcentcolon \zeta \in \varrho(A) \}= \mathcal K.
\end{gather*}
\end{theorem}
\begin{proof}
    We have established the ``only if`` direction in Theorem \ref{Thrm:Main1} and the ``if'' direction in Theorem  \ref{MainConverse}.
\end{proof}
\color{black}
\section{Examples}
\subsection{Generalized Nevanlinna functions}
There are several conceptually different ways to construct minimal realizations for generalized Nevanlinna functions. For example, the construction in \cite{edsjsr.119452920031201} is based on reproducing kernel Pontryagin spaces, while  the one in \cite{DijksmaLangerLugerShondin04}  relies on a certain factorization result.  In this article, we have used the classical theory of model spaces and related function spaces.  This function-theoretic approach allows us to give a short reproof of the following result \cite[Theorem 1.1]{WIETSMA2018997}:
\begin{theorem}
    Let $f$ be a Generalized Nevanlinna function with index $\kappa$. For any given $w \in \C^-$, the equation $f(\zeta)-w=0$ has exactly $\kappa$ solutions on $\C^+$ (counted with multiplicities).
\end{theorem}
\begin{proof}
    Let $f$ be represented as 
    \begin{align*}
        f = \mathrm{i}  \frac{h_2-h_1}{h_2+h_1},
    \end{align*}
    where $h_1\in \mathcal{S}_0$ and $h_2\in \mathcal{S}_0$ are as in Proposition \ref{Prop:MainTech}.
    The index $\kappa$ coincides with the dimension of the negative part of $\mathcal{L}(f)$, which means that  $\kappa =\mathrm{dim}(\mathcal{B}(h_2))$. Consequently, the function $h_2 \in \mathcal{S}_0$  is generated by a finite Blaschke product with $\kappa$ zeros see Section \ref{Sec:Bh spaces}.  However, the zeros of $h_2$ in $\C^+$ correspond exactly to the zeros of $f+\mathrm{i}=0$ in $\C^+$(counted with multiplicities). 
    \par\smallskip
    In summary, we have proven our claim for $w=-\mathrm{i}$.
    We extend this to arbitrary points in the lower-half plane in the following way: First, note that  the function $a  f+b, \ a,b \in \R_+$ has the same index as $f$. Therefore,  we can apply  our result for $w=-\mathrm{i}$ to this function, and we infer that the equation
    \begin{align*}
        (a  f+b)+\mathrm{i}=0 
        \Leftrightarrow
        f - \frac{\mathrm{i}+b}{-a}=0, \quad a,b \in \R_+
    \end{align*}
     has exactly $\kappa$ solutions on $\C^+$ (counted with multiplicities). Finally, it is clear that we can write any $w \in \C^-$ as $w=\frac{\mathrm{i}+b}{-a}$ with $a,b \in \R_+$.
\end{proof}
\subsection{Real complex functions}
Next we discuss real complex function which  have 
 attracted substantial interest in the recent past, see \cite{MR213787120050101}, \cite{edsjsr.2490349820030101} and the survey \cite{articleRealpositive}. They are defined as follows:
\begin{definition}
    Let $f \in \mathcal{N}_{sym}$. Then  $f$ is  a real complex function function if $f_{|\C^+}$ has real boundary values $\lambda$-almost eeverywhere 
\end{definition}
We have seen in  Section \ref{Sec:Intersection} that  such a function $f$ can be represented as 
\begin{align*}
    f = \mathrm{i}  \frac{h_2-h_1}{h_2+h_1},
\end{align*}
where $h_1$ and $h_2$ are generated by inner functions $V_1$ and $V_2$. In this case, the associated model spaces $\mathcal{H}(V_1)$ and $\mathcal{H}(V_2)$ intersect trivially, since  $h_1$ and $h_2$ are relatively prime and in turn also $V_1$ and $V_2$ are relatively prime. Consequently, the restriction operator 
\begin{align*}
    \mathrm{Res} \vcentcolon\mathcal{B}(h_1) \oplus  -\mathcal{B}(h_2) \rightarrow \mathcal{H}(V_1) \oplus -\mathcal{H}(V_2), \quad f \mapsto f_{|\C^+}
\end{align*}
is an isometric isomorphism. This means that its   realization acts, up to an isomorphism, on a direct outer sum of Model spaces. As already mentioned, the structure of model spaces has been the subject of extensive research and is very well understood.
\par\smallskip
Here, we want to give one example of how our construction takes the structure of the underlying function $f$ into account. If $f$ is a real complex function, then 
 $f_{|\C^+}$ and $f_{|\C^-}$ are pseudocontinuations of each other, and this  property is inherited by the space $\mathcal{L}(f)$. Indeed, since this is true for any function $g \in \mathcal{B}(h_1) \oplus  -\mathcal{B}(h_2)$ and for the function
 \begin{align*}
    \frac{1}{h_1+h_2},
 \end{align*}
 it is also true for the function
 \begin{align*}
     \frac{1}{h_1+h_2} \cdot g \in  \frac{1}{h_1+h_2} \cdot (\mathcal{B}(h_1) \oplus  -\mathcal{B}(h_2)) = \mathcal{L}(f).
 \end{align*}
\subsection{The extended Nevanlinna class}
Originally, we were interested in the following function class, which was  introduced in  \cite{delsarte1986pseudo}:
\begin{definition}
    Let $f \in \mathcal{N}_{sym}$. Then $f$ is in the extended Nevanlinna class if $f$ is a product of the form
    $f=g \cdot m$, where $g$ is a Herglotz-Nevanlinna function and $m$ is a density function. Here, a density function is a real complex function with  non-negative boundary values.    
\end{definition}
The extended Nevanlinna class generalizes the notion of generalized Nevanlinna functions in a natural way by allowing the density factor $m$ to be non-rational. On the other hand, this class also generalizes the class of real complex functions. Indeed, it was shown in \cite[Theorem 3.1]{delsarte1986pseudo} that any real complex function $f$ is a product of the form 
\begin{align*}
    f=g \cdot m
\end{align*}
where $g$ is a Herglotz-Nevanlinna function with a real boundary values and $m$ is a density function. 
In any case, we can represent $f$ as 
\begin{align*}
    f = \mathrm{i}  \frac{h_2-h_1}{h_2+h_1}
\end{align*}
where $h_2$ is generated by an inner function. 
\newpage

\printbibliography

\end{document}